\documentclass[12pt]{article}

\usepackage{amsfonts,graphicx,amsmath,amssymb,amsthm,geometry,mathtools,authblk,mathrsfs,bbm,hyperref,nicefrac,enumerate,comment,dsfont}

\newcommand{\closure}[1]{\overline{#1}}
\newcommand{\norm}[1]{\left \lVert#1\right \rVert}

\newtheorem{lemma}{Lemma}
\newtheorem{remark}[lemma]{Remark}

\newtheorem{theorem}[lemma]{Theorem}
\newtheorem{corollary}[lemma]{Corollary}
\newtheorem{definition}[lemma]{Definition}

\begin{document}
	
	\title{Perturbation theory and linear partial differential equations with delay}
	
	\author{}

	\author{Ismail T. Huseynov$^1$, Nazim I. Mahmudov$^1$
		\bigskip
		\\
		\small{$^1$Department of Mathematics, Eastern Mediterranean University, Mersin 10, 99628,\\ Turkish Republic of Northern Cyprus, Turkey,\\ e-mail adresses: $ismail.huseynov@emu.edu.tr$, $nazim.mahmudov@emu.edu.tr$}
		
		\smallskip
	}
	
	\maketitle
	\begin{abstract}
	\noindent Functional evolution equations are used in the modeling of numerous physical processes. In this work, our main tool is perturbation theory of strongly continuous semigroups. The advantage of this technique is that one can provide functional evolution equations with the explicit representation formulas of the solution. First, we introduce a closed form of the fundamental solution of the evolution equation with a discrete delay using the delayed Dyson-Phillips series. Then we set up the analytical representation formulas of the classical solutions of linear homogeneous/non-homogeneous evolution equations with a constant delay in a Banach space. In the special case, when a strongly continuous group $\left\lbrace \mathcal{T}(t)\right\rbrace_{t\in \mathbb{R}}$ commutes with a bounded linear operator $A_{1}$, we obtain an elegant formula for the fundamental solution using the powers of the resolvent operator of $A_{0}$.
	Furthermore, we consider delay evolution equations with permutable/non-permutable linear bounded operators and derive crucial results in terms of non-commutative analysis. Finally, we present an example, in the context of a one-dimensional heat equation with a discrete delay to demonstrate the applicability of our theoretical results and give some comparisons with existing results.
	\end{abstract}

	\tableofcontents

\section{Introduction}
 \label{Sec:intro}
 Differential equations with time delays are called \textit{functional differential equations} or \textit{time-delay systems}. A distinguishing feature of the functional differential equations under consideration is that the evolution rate of the processes described by such equations depends on the \textit{pre-history}. Time-delay systems are central to all areas of science, particularly the physical and biological sciences, such as immunology, medicine, nuclear power generation, heat transfer, track signal processing, regulation systems, etc. (see, \cite{Kolmanovskii-Myshkis}-\cite{Bellman-Cooke}).
 
 One of the main problems of qualitative theory for functional differential equations is the explicit representation formulas for the solutions of linear delay differential equations in finite and infinite dimensional spaces. Let us briefly summarise what has been done to solve the problem of representing solutions of linear time-delay systems in $\mathbb{R}^{n}$ by \textit{delayed matrix-valued functions}. We mention the pioneering work \cite{Khusainov-ordinary-1,Khusainov-ordinary-2} in which the first results for linear time-delay systems were found. In \cite{Khusainov-ordinary-1}, Khusainov and Shuklin studied the problem of relative controllability for a linear control system with a single constant delay using the \textit{pure delayed exponential matrix function}. In \cite{Khusainov-ordinary-2}, Khusainov et al. proposed an explicit representation formula of the solution for a linear delay differential system with \textit{permutable} matrix coefficients.
 Furthermore, in \cite{huseynov-mahmudov,mahmudov}, the classical results are extended to time-delay systems of fractional order with \textit{permutable} \cite{huseynov-mahmudov} and \textit{non-permutable} \cite{mahmudov} matrices using \textit{delayed perturbation of Mittag-Leffler matrix functions}.

 In recent years, there has been increasing interest in the study of partial differential equations with delay\cite{Khusainov-partial-1}-\cite{Samoilenko}. Partial differential equations of the parabolic type with a delayed argument are widely used to model and study various problems that arise in the study of population dynamics in ecological systems (see, \cite{ecology}). In \cite{Khusainov-partial-1,Khusainov-Pokojovy-Racke}, Khusainov et al. studied the existence and uniqueness of a classical solution to the initial-boundary value problem for the one-dimensional heat equation with a constant delay. For the construction of a solution, the authors in \cite{Khusainov-partial-1,Khusainov-Pokojovy} used the method of separation of variables or the Fourier method. In \cite{Samoilenko}, Samoilenko and Serheeva have described an algorithm for the construction of global solutions of the delayed one-dimensional heat equation with time-varying coefficients.
 
 \textit{Perturbation theory} for \textit{strongly continuous operator families} is a useful tool for \textit{evolution equations}, in particular, for \textit{partial differential equations} in modeling many physical phenomena \cite{philips,perturbation,per-frac-2}. The perturbation of linear operators in a Banach space has been studied to a considerable extent, most notably by Phillips \cite{philips}, Travis \& Webb \cite{travis-webb}, and Lutz \cite{lutz}. In \cite{philips}, Phillips first studied the implication for a linear abstract Cauchy problem with the infinitesimal generator $A_{0}:\mathcal{D}\left( A_{0}\right)\subseteq X\to X$ of a \textit{strongly continuous semigroup} which is conserved under bounded perturbation $A_{1}\in\mathcal{L}\left(X\right)$ in a Banach space $X$:
 \begin{equation}\label{klassik}
 	\begin{cases}
 		\frac{\mathrm{d}}{\mathrm{d}t}u(t)=\left( A_{0}+A_{1}\right) u(t), \quad t \geq 0,\\
 		u(0)=x \in \mathcal{D}\left(A_{0} \right).
 	\end{cases}
 \end{equation}
 In \cite{travis-webb}, Travis and Webb have established sufficient conditions for \textit{perturbed strongly continuous cosine operator families}. In \cite{lutz}, Lutz has intended to perturb the infinitesimal generator $A_{0}$ by adding to it a linear time-varying bounded operator $A_{1}(\cdot): \mathbb{R} \to \mathcal{L}\left(X\right)$ and investigating some \textit{perturbation properties} for infinitesimal generators of \textit{strongly continuous cosine and sine operator families}. In terms of fractional sense, some perturbation results for fractional abstract initial value problems of order $1<\alpha < 2$ in \cite{per-frac-2,bazhlekova} have been studied in several aspects and the obtained results agree with the classical ones when $\alpha=2$.
 
 In the general case, \textit{partial differential equations with delay} can be written in an \textit{abstract way} as follows:
 \begin{equation}\label{deviate}
 	\begin{cases}
 		\frac{\mathrm{d}}{\mathrm{d}t}u(t)=A_{0}u(t)+A_{1}u_{t}, \quad t \geq 0,\\
 		u(0)=x,\\
 		u_{0}=f,
 	\end{cases}
 \end{equation}	
 where
 \begin{itemize}
 	\item $x\in X$, $X$ is a Banach space;
 	\item $A_{0}: \mathcal{D}\left(A_{0} \right) \subseteq X \to X$ is a closed and densely-defined linear operator;
 	\item $A_{1}: \mathbb{L}^{p}\left( [-\tau,0],X\right) \to X$ is the \textit{delay operator} which is a bounded linear operator;
 	\item $u(\cdot):[-\tau,\infty) \to X$ and $u_{t}(\cdot):[-\tau,0]\to X$ is the \textit{history function} defined by $u_{t}(s)=u(t+s)$ for $s\in [-\tau,0]$;
 	\item $f(\cdot): \mathbb{L}^{p}\left([-\tau,0], X\right)$ for $1\leq p<\infty$.
 \end{itemize}
 To analyse such equations, the first step is to choose a suitable \textit{state space}. There are two different directions to solve this kind of equations. One way is to apply the \textit{semigroup theoretical methods} in the state space $\mathcal{X}=\mathbb{C}\left([-\tau,0],X\right) $. In this case, the relation between the solutions of the delay equations \eqref{deviate} and a corresponding \textit{translation semigroup} has been widely studied (see, \cite{engel-nagel}, Section VI.6).
 
 In the second direction, Bátkai and Piazzera in \cite{batkai-piazzera-partial,batkai-book} have studied the \textit{equivalency} between the partial differential equation with delay \eqref{deviate} and the following abstract Cauchy problem, associated with the operator $\left(\mathcal{A},\mathcal{D}\left( \mathcal{A}\right) \right)$:
 \begin{equation}\label{deviate-1}
 	\begin{cases}
 		\frac{\mathrm{d}}{\mathrm{d}t}U(t)=\mathcal{A}U(t), \quad t \geq 0,\\
 		U(0)=\binom{x}{f}\in \mathcal{D}\left( \mathcal{A}\right),
 	\end{cases}
 \end{equation}
 on the \textit{product} state space $\mathcal{X}\coloneqq X\times \mathbb{L}^{p}\left([-\tau,0],X\right)$ with $U(t)=\binom{u(t)}{u_t}\in \mathcal{X}$ for $t\geq 0$, and a linear operator $\mathcal{A}=\begin{pmatrix}
 	A_{0} \quad  A_{1} \\ 0 \quad  \frac{\mathrm{d}}{\mathrm{d}s}  
 \end{pmatrix}$.
 Then, in \cite{batkai-piazzera-partial} using the \textit{Miyadera-Voigt's perturbation theorem}, the authors  give sufficient conditions for
 \begin{equation*}
 	\mathcal{A}=\mathcal{A}_{0}+\mathcal{A}_{1}=\begin{pmatrix}
 		A_{0} \quad  0 \\ 0 \quad  \frac{\mathrm{d}}{\mathrm{d}s}  
 	\end{pmatrix}+\begin{pmatrix}
 		0\quad  A_{1} \\ 0 \quad  0
 	\end{pmatrix}
 \end{equation*} 
 to be the infinitesimal generator of a \textit{$C_{0}$-semigroup} on a Banach space $\mathcal{X}$.  Note that this reformulation of the delay evolution equation \eqref{deviate} has the advantage that the question of its well-posedness is reduced to the question of whether or not the \textit{operator matrix} $\mathcal{A}$ generates a \textit{strongly continuous semigroup} on the Banach space $\mathcal{X}$.

 Motivated by Phillips \cite{philips} and Bátkai \& Piazzera \cite{batkai-piazzera-partial}, we consider the following abstract Cauchy problem for an evolution equation with a discrete delay in a Banach space $X$:
 \begin{equation} \label{non-homogeneous}
 	\begin{cases}
 		\frac{\mathrm{d}}{\mathrm{d}t}u(t)=A_{0}u(t)+A_{1}u(t-\tau)+g(t), \quad  t\geq 0,\\ 
 		u(t)=\varphi(t), \quad -\tau \leq t \leq 0,
 	\end{cases}
 \end{equation}
 where $A_{0}:\mathcal{D}\left( A_{0}\right) \subseteq X \to X$ is a unbounded linear operator, $A_{1} \in \mathcal{L}\left( X\right)$ and $\tau>0$ is a fixed time delay. Moreover, an initial function $\varphi(\cdot):[-\tau,0]\to X$ is describing the \textit{prehistory of the system} and a forcing term $g(\cdot):[0,\infty)\to X$ is describing the \textit{external forces of the system}.
 
 Unlike the authors in \cite{batkai-piazzera-partial}, we consider an abstract differential equation with a discrete delay \eqref{non-homogeneous} on the state space $X$ rather than the product state space.
 In this work, our main tool is the \textit{perturbation theory} for the strongly continuous semigroups and the advantage of this technique is that one can provide the functional evolution equation \eqref{non-homogeneous} with the explicit representation formulas of the solution. Exploiting the \textit{perturbation theorem} of Phillips \cite{philips}, we derive the closed-form of a \textit{fundamental solution} $ \mathcal{S}(t;\tau)$, $t\geq -\tau$ via a \textit{delayed Dyson-Phillips series}. Moreover, this result agrees with the \textit{delayed perturbation of an operator-valued exponential function} for a delay evolution equation with bounded linear operator coefficients. As an application of our theoretical results, we derive an explicit representation formula of the classical (mild) solution of the one-dimensional heat equation with a discrete delay and give some comparisons between the closed-form solutions and existing works such as \cite{Khusainov-partial-1,jde}.
 
 The paper includes significant updates in the \textit{theory of abstract delay differential equations} and is structured as follows. Section \ref{sect:prel} is a preparatory section in which we recall the main definitions and results from functional analysis, mainly, operator theory and evolution equations. Section \ref{sect:main} is devoted to the closed-form of a fundamental solution of the abstract Cauchy problem for a linear homogeneous delay evolution equation \eqref{main equation} in terms of a delayed Dyson-Phillips series. 
 In the special case, when a strongly continuous group $\left\lbrace \mathcal{T}(t)\right\rbrace_{t\in \mathbb{R}}$ commutes with a bounded linear operator $A_{1}$, we obtain an elegant formula for the fundamental solution using powers of the resolvent operator of $A_{0}$.
 Moreover, we provide the analytical representation formulas of classical solutions to linear homogeneous and non-homogeneous evolution equations with a discrete delay.
 In Section \ref{sect:bounded}, we consider functional evolution equations with linear bounded operators and some special cases of them. Thus, we propose closed-form solutions for delay evolution equations with permutable and non-permutable linear bounded operators.
 Finally, Section \ref{sect:partial} is devoted to the presentation of an illustrative example on one-dimensional heat equation with delay to show the efficiency and validity of the theoretical results, and we give some important comparisons with the existing works.
 
	\section{Mathematical description}\label{sect:prel}
	We embark on this section by briefly presenting some notations and fundamental preliminaries from functional analysis, especially, operator theory and abstract differential equations \cite{engel-nagel,arendt-batty,pazy,curtain-zwart} which are used in this paper.
	
	\textbf{Some notations}.
	Let us fix some notations. Let $X$ be a complex Banach space with norm $\norm{\cdot}$ and $\mathcal{L}\left(X\right)$ be the Banach algebra of all bounded linear operators on $X$ to itself. The identity and zero (null) operators on $X$ are denoted by $I\in \mathcal{L}\left(X \right)$ and $\Theta\in \mathcal{L}\left(X \right)$, respectively.
	
	Let $\mathbb{N}_{0}\coloneqq \mathbb{N}\cup \left\lbrace 0\right\rbrace$, $\mathbb{R}$ and $\mathbb{C}$ be the set of real and complex numbers, respectively,  $[-\tau,T]=[-\tau,0]\cup \bigcup\limits_{i=0}^{n}(i\tau,(i+1)\tau]\subset \mathbb{R}$ and $[0,T]=\left\lbrace 0\right\rbrace \cup \bigcup\limits_{i=0}^{n}(i\tau,(i+1)\tau]\subset \mathbb{R}$ where $T=(n+1)\tau$ is a pre-fixed positive number for a fixed $n\in \mathbb{N}_{0}$ and $\tau>0$.
	
	In the representation of \textit{delayed} operator-valued functions, we will use the following \textit{indicator} or \textit{characteristic} function $\mathds{1}_{t\geq n\tau}:\mathbb{R}\to\left\lbrace 0,1\right\rbrace$ defined by:
	
	\begin{equation*}
		\mathds{1}_{t\geq n\tau}= \begin{cases}
			1, \quad t\geq n\tau, \\
			0, \quad t<n\tau,
		\end{cases}
	\end{equation*}
	where $n\in \mathbb{N}_{0}$ and $\tau>0$.
	
	Alternatively, for the \textit{delayed} operator-valued functions, we will also use the \textit{ramp} or \textit{positive part} function $(t)_{+}:\mathbb{R}\to [0,\infty)$ defined by the formula:
	\begin{equation*}
		\left( t\right)_{+}=\max\left( t,0\right)=\begin{cases}
			t, \quad t\geq 0,\\
			0, \quad t<0.
		\end{cases}
	\end{equation*}

	\textbf{Semigroup theory.} \textit{$C_{0}$-semigroups} serve to describe the time evolution of autonomous linear systems. We provide some fundamental facts concerning $C_{0}$-semigroups, their generators and their applications to the \textit{abstract differential equations} in a Banach space $X$ which are used throughout this paper.
	
	We shall be concerned with a family of bounded linear operators $\left\lbrace \mathcal{T}(t)\right\rbrace_{t\geq 0}$ on a half-line $[0,\infty)$ to $\mathcal{L}\left(X\right)$ is called a \textit{$C_{0}$-semigroup} or a\textit{ strongly continuous semigroup} satisfying the following hypotheses:
	\begin{itemize}
		\item $\mathcal{T}(0)=I$;
		
		\item $\mathcal{T}(t+s)=\mathcal{T}(t)\mathcal{T}(s)$ for all $t,s \geq 0$;
		
		\item $\lim\limits_{t\to 0_{+}}\norm{\mathcal{T}(t)x-x}=0$ for every $x \in X$.
	\end{itemize}
	Note that the condition $(iii)$ is a condition of \textit{strong continuity} of the function $t\mapsto \mathcal{T}(t)$ at point $t=0$. Alternatively, the last condition can be given as follows:
	\begin{itemize}
		\item $\mathcal{T}(t)x$ is continuous in $t$ on $[0,\infty)$ for each $x\in X$.
	\end{itemize}
	If $\mathcal{T}(t)$ is defined for $t\in \mathbb{R}$, the second condition holds for all $t,s \in \mathbb{R}$ and the function $t\to \mathcal{T}(t)x$ is continuous with respect to $t$ on $\mathbb{R}$ for every $x\in X$, then $\left\lbrace   \mathcal{T}(t)\right\rbrace_{t\in \mathbb{R}}$ is called a \textit{$C_{0}$-group}. 
	
	A linear operator $A_{0}:X \to X$ defined by
	\begin{equation*}
		A_{0}x=\lim_{t \to 0_{+}} \frac{\mathcal{T}(t)x-x}{t}, \quad \text{for every} \quad x \in \mathcal{D}(A_{0}),
	\end{equation*}
	where
	\begin{equation*}
		\mathcal{D}(A_{0})=\left\lbrace x \in X: \quad \lim_{t \to 0_{+}} \frac{\mathcal{T}(t)x-x}{t} \quad \textit{exists} \right\rbrace,
	\end{equation*}
	is the \textit{infinitesimal generator} of the $C_{0}$-semigroup $\left\lbrace \mathcal{T}(t)\right\rbrace_{t\geq 0}$, defined on its domain $\mathcal{D}(A_{0})$.
	It is known that the \textit{infinitesimal generator} $A_{0}:\mathcal{D}(A_{0}) \subseteq X \to X$ of a strongly continuous semigroup $\left\lbrace \mathcal{T}(t)\right\rbrace_{t\geq 0}$ is a \textit{closed} and \textit{densely-defined} linear operator on $X$, i.e., $\closure{\mathcal{D}(A_{0})}=X$.
	
	For a strongly continuous semigroup $\left\lbrace \mathcal{T}(t)\right\rbrace_{t\geq 0}$, the following assertions are equivalent:
	\begin{itemize}
		\item  If $x \in \mathcal{D}(A_{0})$, then $\mathcal{T}(t)x \in \mathcal{D}(A_{0})$ and
		\begin{equation*}
			\frac{\mathrm{d}}{\mathrm{d}t}\mathcal{T}(t)x=A_{0}\mathcal{T}(t)x=\mathcal{T}(t)A_{0}x, \quad t\geq 0;
		\end{equation*}
		\item For every $x \in X$, one has
		\begin{equation*}
			\int_{0}^{t}\mathcal{T}(s)x\mathrm{d}s \in \mathcal{D}(A_{0}), \quad t\geq 0;
		\end{equation*}
		\item For every $t \in [0,\infty)$, the following identities hold true:
		\begin{align*}
			\mathcal{T}(t)x-x&=A_{0}\int_{0}^{t}\mathcal{T}(s)x\mathrm{d}s, \quad x \in X,\\
			&=\int_{0}^{t}\mathcal{T}(s)A_{0}x\mathrm{d}s, \quad x \in \mathcal{D}(A_{0}).
		\end{align*}
	\end{itemize} 
	
	It is known that for every strongly continuous semigroup $\left\lbrace \mathcal{T}(t)\right\rbrace_{t \geq 0}$, there exists constants $\omega \in\mathbb{R}$ and $M\geq 1$ such that
	\begin{equation}\label{exponentiall}
		\norm{\mathcal{T}(t)} \leq M \exp\left( \omega t\right) , \quad  t\geq 0.
	\end{equation}
	The \textit{resolvent} $\mathcal{R}(\lambda_0; A_{0})=\left( \lambda_{0}I-A_{0}\right)^{-1}$ of $A_{0}$ is defined on the \textit{resolvent set} $\rho(A_{0})$ for $A_{0}$:
	\begin{equation*}
		\rho(A_{0})=\left\lbrace \lambda_{0}\in \mathbb{C}:  \lambda_{0}I-A_{0}:\mathcal{D}\left( A_{0}\right) \to X \quad \textit{is bijective} \right\rbrace
	\end{equation*}
	and belongs to $\mathcal{L}\left(X\right)$.
	The \textit{resolvent} $\mathcal{R}(\lambda_0;A_{0})$ of $A_{0}$ satisfies the following identities:
	\begin{align*}
		&\left( \lambda_{0}I-A_{0}\right)\mathcal{R}(\lambda_0; A_{0})=I,\\
		&\mathcal{R}(\lambda_0; A_{0})\left( \lambda_{0}I-A_{0}\right)x=x, \quad x\in \mathcal{D}\left(A_{0}\right). 
	\end{align*}
	Moreover, if $Re(\lambda_0)>\omega$, then $\lambda_{0}\in \rho(A_{0})$, and 
	\begin{equation}\label{resolvent}
		\mathcal{R}(\lambda_0; A_{0})x=\int_{0}^{\infty}\exp\left(-\lambda_{0} t\right) \mathcal{T}(t)x\mathrm{d}t, \quad x \in X.
	\end{equation}
	By the formulae \eqref{exponentiall} and \eqref{resolvent}, we derive the following estimation:
	\begin{equation*}
		\norm{\mathcal{R}(\lambda_0; A_{0})}\leq \frac{M}{Re(\lambda_{0})-\omega}.
	\end{equation*}	
	
	In the particular case, if the \textit{infinitesimal generator} $A_{0}:X\to X$ is \textit{bounded}, i.e., there exists $M>0$ such that $\norm{A_{0}x}\leq M\norm{x}$ for all $x\in \mathcal{D}(A_{0})$, then the following relations hold true:
	\begin{itemize}
		\item The domain $\mathcal{D}(A_{0})$ is all of $X$, i.e., $\mathcal{D}(A_{0})=X$;
		
		\item The $C_{0}$-semigroup $\left\lbrace \mathcal{T}(t)\right\rbrace_{t\geq 0}$ is \textit{uniformly continuous} at the point $t=0$:
		\begin{equation*}
			\lim\limits_{t\to 0_{+}}\norm{\mathcal{T}(t)-I}=0.	
		\end{equation*}
		\item The $C_{0}$-semigroup is defined by the formula
		\begin{equation*}
			\mathcal{T}(t)=\exp\left( A_{0}t\right) =\sum_{k=0}^{\infty}A_{0}^{k}\frac{t^{k}}{k!}, \quad t \geq 0.
		\end{equation*}
	\end{itemize}
	
	Consider the following abstract Cauchy problem for a linear homogeneous \textit{evolution} equation in a Banach space $X$:
	\begin{equation}\label{abstract}
		\begin{cases}
			u^{\prime}(t)=A_{0}u(t), \quad t \geq 0,\\
			u(0)=x,
		\end{cases}
	\end{equation}
	where $A_{0} : \mathcal{D}(A_{0}) \subseteq X \to X$ is an infinitesimal generator of $C_{0}$-semigroup $\left\lbrace \mathcal{T}(t)\right\rbrace_{t \geq 0}$. 	
	Then,
	\begin{itemize}
		\item for every $x \in \mathcal{D}(A_{0})$, the function
		\begin{equation*}
			u(\cdot): [0,\infty) \ni t \mapsto u(t)=\mathcal{T}(t)x \in \mathcal{D}(A_{0})
		\end{equation*}
		is the unique \textit{classical} or \textit{strong} solution of \eqref{abstract} with initial value $x$.
		\item for every $x \in X$, the function
		\begin{equation*}
			u(\cdot): [0,\infty) \ni t \mapsto u(t)=\mathcal{T}(t)x \in X
		\end{equation*}
		is the unique \textit{mild} solution of \eqref{abstract} with initial value $x$.
	\end{itemize}	
	
	\textbf{Perturbation theory}. In many concrete situations, the evolution equation is given as a sum of several terms having different physical meaning and different mathematical properties. In such situations \textit{perturbation theory} plays a very useful tool in the hands of both analyst and physicist. In this part, we provide some fundamental properties on \textit{bounded time-independent perturbations} of $C_{0}$-semigroups (i.e., generator $A_{0}$ \textit{perturbed} by the \textit{bounded} linear operator $A_{1}$).
	
	If $A_{0}:\mathcal{D}(A_{0}) \subseteq X \to X$ is an infinitesimal generator of $C_{0}$-semigroup $\left\lbrace \mathcal{T}(t)\right\rbrace_{t \geq 0}$ and $A_{1}\in \mathcal{L}\left(X\right)$, then the semigroup of bounded linear operators $\left\lbrace \mathcal{S}(t)\right\rbrace_{t\geq 0}$ generated by  $A_{0}+A_{1}$ defined on $\mathcal{D}\left(A_{0}\right)$ can be represented by an \textit{absolutely} and \textit{uniformly} (in every \textit{compact} subset of $[0,\infty)$) \textit{convergent} series as follows:
	\begin{equation}\label{S}
		\mathcal{S}(t)=\sum_{n=0}^{\infty}S_{n}(t), \quad t\geq0,
	\end{equation} 
	where 
	\begin{align*}
		S_{0}(t)&=\mathcal{T}(t),\quad t\geq0,\\
		S_{n}(t)&=\int_{0}^{t}\mathcal{T}(t-s)A_{1}S_{n-1}(s)\mathrm{d}s,\\
		&=\int_{0}^{t}S_{n-1}(t-s)A_{1}\mathcal{T}(s)\mathrm{d}s,\quad n\in \mathbb{N},\quad t\geq 0.
	\end{align*}
	Moreover, the strongly continuous semigroup $\left\lbrace \mathcal{S}(t)\right\rbrace_{t\geq 0}$ for some $\omega \in\mathbb{R}$ and $M\geq 1$ is satisfying
	\begin{equation*}
		\norm{\mathcal{S}(t)}\leq M\exp( \omega_{1}t), \quad \omega_{1}\coloneqq \omega+M\norm{A_{1}}, \quad t\geq 0.
	\end{equation*}
	For $Re(\lambda_{0})>\omega_{1}$, we have
	\begin{equation*}
		\mathcal{R}(\lambda_0; A_{0}+ A_{1})x=\int_{0}^{\infty}\exp\left(-\lambda_{0} t\right) \mathcal{S}(t)x\mathrm{d}t, \quad x \in X.
	\end{equation*}
	If $A_{0}$ is an infinitesimal generator of $C_{0}$-semigroup $\left\lbrace \mathcal{T}(t)\right\rbrace_{t \geq 0}$ and $A_{1}$ is a bounded linear operator, then the \textit{resolvent} $\mathcal{R}(\lambda_0; A_{0}+ A_{1})$ of $A_{0}+ A_{1}$ satisfies the following identities:
	\begin{align*}
		\mathcal{R}(\lambda_0; A_{0}+ A_{1})x&=\sum_{n=0}^{\infty}\mathcal{R}(\lambda_0; A_{0})\Big[A_{1}\mathcal{R}(\lambda_0;A_{0})\Big]^{n}x\\
		&=\sum_{n=0}^{\infty}\Big[A_{1}\mathcal{R}(\lambda_0;A_{0})\Big]^{n}\mathcal{R}(\lambda_0; A_{0})x, \quad x\in X.
	\end{align*}
	
	Consider the following abstract Cauchy problem for a linear non-homogeneous evolution equation in a Banach space $X$:
	\begin{equation}\label{main-equation}
		\begin{cases}
			\frac{\mathrm{d}}{\mathrm{d}t}u(t)=A_{0}u(t)+A_{1}u(t)+g(t), \quad t\geq 0, \\
			u(0)=x.
		\end{cases}
	\end{equation}
	Let $A_{0}:\mathcal{D}\left(A_{0} \right)\subseteq X\to X$ be the infinitesimal generator of a $C_{0}$-semigroup $\left\lbrace \mathcal{T}(t)\right\rbrace_{t\geq 0}$, $A_{1}\in \mathcal{L}\left(X\right)$ and $g(\cdot)\in \mathbb{C}^{1}\left([0,\infty),X\right)$. Then, for each $x\in \mathcal{D}\left(A_{0}\right)$, there exists a unique continuously differentiable solution $u(\cdot):[0,\infty) \to X$ of \eqref{main-equation} which is satisfying $u(t)\in \mathcal{D}\left(A_{0}\right)$ for $t\in [0,\infty)$  with $u(0)=x$. This solution has a closed form:
	\begin{equation*}
		u(t)=\mathcal{S}(t)x+\sum_{n=0}^{\infty}w_{n}(t), \quad t\geq 0,
	\end{equation*}
	where $\mathcal{S}(t)$ is given by \eqref{S}, and 
	\begin{align*}
		w_{0}(t)&=\int_{0}^{t}\mathcal{T}(t-s)g(s)\mathrm{d}s, \quad t\geq 0,\\
		w_{n}(t)&=\int_{0}^{t}\mathcal{T}(t-s)A_{1}w_{n-1}(s)\mathrm{d}s,\\
		&=\int_{0}^{t}w_{n-1}(t-s)A_{1}\mathcal{T}(s)\mathrm{d}s,\quad n\in \mathbb{N},\quad t\geq 0.
	\end{align*}

\section{Main results: a delayed Dyson-Phillips series}\label{sect:main} 
In this section, first, we consider the following abstract Cauchy problem for a linear homogeneous evolution equation with a discrete delay $\tau>0$ in a Banach space $X$:
\begin{equation}\label{main equation}
	\begin{cases}
		\frac{\mathrm{d}}{\mathrm{d}t}u(t)=A_{0}u(t)+A_{1}u(t-\tau), \quad t\geq 0, \\
		u(t)=\varphi(t) \in X, \quad -\tau\leq t \leq 0,
	\end{cases}
\end{equation}
where $A_{0}:\mathcal{D}\left(A_{0} \right)\subseteq X\to X$ is the infinitesimal generator of a $C_{0}$-semigroup $\left\lbrace \mathcal{T}(t)\right\rbrace_{t\geq 0}$ of bounded linear operators, $A_{1}\in \mathcal{L}\left(X\right)$ and the initial function $\varphi(\cdot)\in \mathbb{C}\left([-\tau,0],X\right)$.

Our main aim is to determine a \textit{closed-form} of the \textit{classical (strong)} solution of an abstract Cauchy problem \eqref{main equation}.

\textit{A classical (strong) solution of the abstract Cauchy problem \eqref{main equation} is understood as an operator-valued function $u(\cdot):[-\tau,\infty)\to X$ -continuously defined for all $t\geq -\tau$, continuously differentiable for all $t\geq 0$, $u(t)\in\mathcal{D}\left(A_{0}\right)$ for all $t\geq 0$ with $\varphi(0)\in \mathcal{D}\left(A_{0}\right)$ and satisfying \eqref{main equation} for all $t\geq 0$}.

Meanwhile, this definition can be generalized to \textit{mild} sense as follows:

\textit{A mild solution of the abstract Cauchy problem \eqref{main equation} is understood as an operator-valued function $u(\cdot):[-\tau,\infty)\to X$ -continuously defined for all $t\geq -\tau$, continuously differentiable for all $t\geq 0$, $u(t)\in X$ for all $t\geq 0$ and satisfying \eqref{main equation} for all $t\geq 0$.}

The \textit{fundamental solution} of the abstract initial value problem \eqref{main equation} which is the main part of the solution of delay evolution equation can be defined as follows.
\begin{definition}
	If $\mathcal{S}(\cdot;\tau):[-\tau,\infty)\to \mathcal{L}(X)$ satisfies the following linear homogeneous abstract differential equation with linear operator coefficients:
	\begin{equation}\label{fundamental}
		\frac{\mathrm{d}}{\mathrm{d}t}\mathcal{S}(t;\tau)=A_{0}\mathcal{S}(t;\tau)+A_{1}\mathcal{S}(t-\tau;\tau), \quad t\geq 0,  
	\end{equation}
	under initial conditions 
	\begin{equation}\label{unit}
		\mathcal{S}(t;\tau)=\begin{cases} 
			\Theta,\quad -\tau \leq t <0,\\
			I, \quad t=0, 
		\end{cases}
	\end{equation}
	then $\mathcal{S}(t;\tau)$, $t \geq -\tau$ is called the corresponding \textit{fundamental solution} of an abstract differential equation \eqref{main equation} with a constant delay.
\end{definition}

\begin{remark}
	The fundamental solution $\mathcal{S}(t;\tau)$, $t\geq-\tau$ also satisfies the following differential equation with operator coefficients and initial conditions
	\begin{equation}\label{fundamental-1}
		\begin{cases}
			\frac{\mathrm{d}}{\mathrm{d}t}\mathcal{S}(t;\tau)=\mathcal{S}(t;\tau)A_{0}+\mathcal{S}(t-\tau;\tau)A_{1}, \quad t\geq 0,\\
			\mathcal{S}(t;\tau)=\begin{cases}
				\Theta, \quad -\tau \leq t <0,\\
				I, \quad t=0.
			\end{cases}
		\end{cases}
	\end{equation}
	This does not mean that $\mathcal{S}(t;\tau)$ commutes individually with the coefficient operators $A_{i}$, $i=0,1$ for any $t\in [0,\infty)$.
\end{remark}
\begin{proof}
	To verify this remark, it is sufficient to compare the Laplace image of the \textit{fundamental solution} as a solution of differential equation with operator coefficients \eqref{fundamental} with the abstract delay differential equation \eqref{fundamental-1}.
\end{proof}
\allowdisplaybreaks
A \textit{fundamental solution} is an operator-valued function with values in $\mathcal{L}(X)$ which can be found with the help of \textit{Laplace transform} technique.
Let $\lambda_{0} \in \rho(A_{0})$. Then,
applying the Laplace integral transform each side of \eqref{fundamental} with initial conditions \eqref{unit} and using integration by substitution, we obtain 
\begin{align*}
	\lambda_{0} \int_{0}^{\infty}\exp(-\lambda_{0} t)\mathcal{S}(t;\tau)x\mathrm{d}t-x
	&=A_{0}\int_{0}^{\infty} \exp(-\lambda_{0} t)\mathcal{S}(t;\tau)x\mathrm{d}t\\
	&+A_{1}\int_{0}^{\infty}\exp(-\lambda_{0} t)\mathcal{S}(t-\tau;\tau)x\mathrm{d}t\\
	&=A_{0}\int_{0}^{\infty}\exp(-\lambda_{0}t)\mathcal{S}(t;\tau)x\mathrm{d}t\\
	&+A_{1}\int_{-\tau}^{\infty}\exp(-\lambda_{0} (t+\tau))\mathcal{S}(t;\tau)x\mathrm{d}t\\
	&=A_{0}\int_{0}^{\infty}\exp(-\lambda_{0}t)\mathcal{S}(t;\tau)x\mathrm{d}t\\
	&+A_{1}\exp(-\lambda_{0} \tau)\int_{0}^{\infty}\exp(-\lambda_{0} t)\mathcal{S}(t;\tau)x\mathrm{d}t, \quad x \in X.
\end{align*}
Therefore, for sufficiently large $Re(\lambda_{0})$, the Laplace transform of a \textit{fundamental solution} is defined by
\begin{align}\label{laplace}
	\int_{0}^{\infty}\exp(-\lambda_{0}t)\mathcal{S}(t;\tau)x\mathrm{d}t=\Big(\lambda_{0} I-A_{0}- A_{1}\exp(-\lambda_{0} \tau)\Big)^{-1}x, \quad x \in X.
\end{align}
Moreover, the right-hand side of \eqref{laplace} is the perturbation of the infinitesimal generator $A_{0}$ with a bounded linear operator $\hat{A}_{1}\coloneqq A_{1}\exp(-\lambda_{0} \tau)$:
\begin{align}\label{laplace-1}
	\Big(\lambda_{0}I-A_{0}- A_{1}\exp(-\lambda_{0} \tau)\Big)^{-1}x=\mathcal{R}
	\left(\lambda_{0};A_{0}+ \hat{A}_{1}\right)x, \quad x \in X.
\end{align}
Therefore, for any $x\in X$, the Laplace transform of a \textit{fundamental solution} of \eqref{main equation} can be determined by the resolvent of $A_{0}+ \hat{A}_{1}$ as follows:
\begin{equation}\label{laplace-2}
	\int_{0}^{\infty}\exp(-\lambda_{0} t)\mathcal{S}(t;\tau)x\mathrm{d}t=\mathcal{R}
	\left(\lambda_{0};A_{0}+ \hat{A}_{1}\right)x, \quad \hat{A}_{1}\coloneqq A_{1}\exp(-\lambda_{0} \tau)\in \mathcal{L}(X).	
\end{equation}

The following lemma is given in more general case for closed linear operators and plays a significant role in the proof of Theorem \ref{main-hom}.
\begin{lemma}\label{lema}
	Let $A_{0}:\mathcal{D}(A_{0})\subseteq X\to X$ be closed linear operator on $X$ and assume $A_{1}\in \mathcal{L}(X)$ is such that $\norm{A_{1}\mathcal{R}\left( \lambda_{0};A_{0}\right) \exp(-\lambda_{0}\tau)}<1$ for some $\lambda_{0}\in \rho(A_{0})$. Then, $A_{0}+\hat{A}_{1}$ where $\hat{A}_{1}\coloneqq A_{1}\exp(-\lambda_{0}\tau)\in \mathcal{L}(X)$ is a closed linear operator with domain $\mathcal{D}(A_{0})$ and $\mathcal{R}\left( \lambda_{0};A_{0}+\hat{A}_{1}\right)$ exists, and the following identity holds true:
	\begin{equation}\label{laplas}
		\mathcal{R}\left( \lambda_{0};A_{0}+\hat{A}_{1}\right)=\sum_{n=0}^{\infty}\mathcal{R}(\lambda_{0};A_{0})\Big[A_{1}\mathcal{R}(\lambda_{0};A_{0})\Big]^{n}\exp(-n\lambda_{0}\tau).
	\end{equation}
\end{lemma}
\begin{proof}
	It is obvious that $A_{0}+\hat{A}_{1}$ with $\hat{A}_{1}= A_{1}\exp(-\lambda_{0}\tau)\in \mathcal{L}(X)$ is a closed linear operator with domain $\mathcal{D}(A_{0})$.
	Since the Neumann series $\sum_{n=0}^{\infty}\Big[A_{1}\mathcal{R}(\lambda_{0};A_{0})\exp(-\lambda_{0}\tau)\Big]^{n}$ converges under the hypotheses $\norm{A_{1} \mathcal{R}(\lambda_{0};A_{0})\exp(-\lambda_{0}\tau)}<1$, we note that
	\begin{align*}
		\mathcal{R} &\equiv \sum_{n=0}^{\infty}\mathcal{R}(\lambda_{0};A_{0})\Big[A_{1}\mathcal{R}(\lambda_{0};A_{0})\Big]^{n}\exp(-n\lambda_{0}\tau)\\
		&=\mathcal{R}(\lambda_{0};A_{0})\sum_{n=0}^{\infty}\Big[A_{1}\mathcal{R}(\lambda_{0};A_{0})\exp(-\lambda_{0}\tau)\Big]^{n}\\
		&=\mathcal{R}(\lambda_{0};A_{0})\Big(I-A_{1}\mathcal{R}(\lambda_{0};A_{0})\exp(-\lambda_{0}\tau)\Big)^{-1},
	\end{align*}
	and hence that
	\begin{align}\label{rel}
		&\Big(\lambda_{0}I-A_{0}- A_{1}\exp(-\lambda_{0}\tau) \Big)\mathcal{R}\nonumber\\
		&=\Big(I- A_{1}\mathcal{R}(\lambda_{0};A_{0})\exp(-\lambda_{0}\tau) \Big)\mathcal{R}(\lambda_{0};A_{0})^{-1}\mathcal{R}\nonumber\\
		&=\Big(I- A_{1}\mathcal{R}(\lambda_{0};A_{0})\exp(-\lambda_{0}\tau)\Big)\Big(I-A_{1}\mathcal{R}(\lambda_{0};A_{0})\exp(-\lambda_{0}\tau)\Big)^{-1}\nonumber\\
		&=I.
	\end{align}
	
	Furthermore, the range of $\mathcal{R}$	is precisely $\mathcal{D}(A_{0})$ since the range of $\Big(I-A_{1}\mathcal{R}(\lambda_{0};A_{0})\exp(-\lambda_{0}\tau)\Big)^{-1}$ is $X$. Thus, for a given $x\in\mathcal{D}(A_{0})$ there exists a $y\in X$ such that $x=\mathcal{R}y$. Therefore, by  \eqref{rel}, we attain that
	\begin{align*}
		&\mathcal{R}\Big(\lambda_{0} I-\Big(A_{0}+ A_{1}\exp(-\lambda_{0}\tau)\Big)  \Big)x\\&=\mathcal{R}\Big(\lambda_{0} I-\Big(A_{0}+ A_{1}\exp(-\lambda_{0}\tau) \Big) \Big)\mathcal{R}y\\&=\mathcal{R}y\\&=x,
	\end{align*}		
	so that $\mathcal{R}$ is both a left and a right inverse. The proof is complete.	
\end{proof}

\begin{remark}
	Note that under the condition	$\norm{\mathcal{R}\left( \lambda_{0};A_{0}\right)A_{1} \exp(-\lambda_{0}\tau)}<1$ for some $\lambda_{0}\in \rho(A_{0})$, the following identity also holds true:
	\begin{equation*}
		\mathcal{R}\left( \lambda_{0};A_{0}+\hat{A}_{1}\right)=\sum_{n=0}^{\infty}\Big[\mathcal{R}(\lambda_{0};A_{0})A_{1}\Big]^{n}\mathcal{R}(\lambda_{0};A_{0})\exp(-n\lambda_{0}\tau).
	\end{equation*}
\end{remark}

The closed-form of a \textit{fundamental solution} to delay evolution equation \eqref{main equation} can be expressed with the help of \textit{a delayed Dyson-Phillips series}.
\begin{theorem}\label{main-hom}
	Let $A_{0}:\mathcal{D}\left( A_{0}\right) \subseteq X\to X$ be the infinitesimal generator of a strongly continuous semigroup of bounded linear operators $\left\lbrace \mathcal{T}(t)\right\rbrace_{t\geq 0}$ and $A_{1} \in \mathcal{L}(X)$. Then there is a unique one-parameter family of bounded linear operators $\mathcal{S}(t;\tau)$ strongly continuous on $[-\tau,\infty)$ such that $\mathcal{S}(t;\tau)=\Theta$, $-\tau \leq t <0$ and $\mathcal{S}(0;\tau)=I$; strongly continuously differentiable on $[0,\infty)$ to $\mathcal{L}\left(X\right) $ and satisfying
	\begin{equation}
		\frac{\mathrm{d}}{\mathrm{d}t}\mathcal{S}(t;\tau)=A_{0}\mathcal{S}(t;\tau)+A_{1}\mathcal{S}(t-\tau;\tau), \quad t\geq 0. 
	\end{equation}
	\allowdisplaybreaks
	This solution has an explicit representation formula:
	\begin{equation}\label{power}
		\mathcal{S}(t;\tau)=\sum_{n=0}^{\infty}S_{n}(t,n\tau)\mathds{1}_{t\geq n\tau}, \quad t\geq 0,
	\end{equation}
	where 
	\begin{align}\label{recursive}
		&S_{0}(t,0)=\mathcal{T}(t),\quad t\geq 0,\nonumber\\
		&S_{n}(t,n\tau)=\int_{n\tau}^{t}\mathcal{T}(t-s)A_{1}S_{n-1}(s-\tau,(n-1)\tau)\mathrm{d}s, \quad t\geq n\tau, \quad n \in \mathbb{N}.
	\end{align}
\end{theorem}
\begin{proof} We divide by the proof some parts:
	
	\textbf{1.} Since a \textit{fundamental solution} is satisfying initial conditions \eqref{unit}, i.e., $\mathcal{S}(t;\tau)=\Theta$, $-\tau \leq t <0$ and $\mathcal{S}(0;\tau)=I$, it is strongly continuous on the initial interval $[-\tau,0]$.
	
	It is obvious that $S_{0}(t,0)=\mathcal{T}(t)$ is strongly continuous on $[0,\infty)$ that $\norm{S_{0}(t,0)} \leq M \exp(\omega t)$ by \eqref{exponentiall}. Suppose $S_{n}(t,n\tau)\mathds{1}_{t\geq n\tau}$ is like-wise strongly continuous on $[0,\infty)$ and that
	\begin{equation}\label{exponential}
		\norm{S_{n}(t,n\tau)} \leq M\Big( M\norm{A_{1}}\exp(-\omega\tau)\Big) ^{n}\frac{(t-n\tau)^{n}}{n!}\exp(\omega t), \quad t\geq n\tau.
	\end{equation}	
	Then $\mathcal{T}(t-s)A_{1}S_{n}(s-\tau,n\tau)$ will be strongly continuous on $[(n+1)\tau,t]$ so that the integral defining  $S_{n+1}(t,(n+1)\tau)\mathds{1}_{t\geq (n+1)\tau}$ exists in the strong topology for $t\in [0,\infty)$, by induction principle. Furthermore, by \eqref{recursive} and \eqref{exponential}, we have:
	\begin{align*}
		\norm{S_{n+1}(t,(n+1)\tau)}&\leq M\norm{A_{1}}\int_{(n+1)\tau}^{t}\exp(\omega(t-s))\norm{S_{n}(s-\tau,n\tau)}\mathrm{d}s\\
		&\leq M \Big( M\norm{A_{1}}\exp(-\omega \tau) \Big)^{n+1}\exp(\omega t)\int_{(n+1)\tau}^{t}\frac{(s-(n+1)\tau)^{n}}{n!}\mathrm{d}s\\
		&=M \Big(M\norm{A_{1}}\exp(-\omega \tau)\Big)^{n+1}\frac{(t-(n+1)\tau)^{n+1}}{(n+1)!}\exp(\omega t),\quad t\geq (n+1)\tau.
	\end{align*}
	Finally, for $t_1<t_2$, we have
	\begin{align}\label{ineq}
		&\norm{S_{n+1}(t_2,(n+1)\tau)x-S_{n+1}(t_1,(n+1)\tau)x}\nonumber\\ \leq& \int_{(n+1)\tau}^{t_1}\norm{\Big( \mathcal{T}(t_2-s)-\mathcal{T}(t_1-s)\Big) A_{1}S_{n}(s-\tau,n\tau)x}\mathrm{d}s\nonumber\\
		+&\int_{t_1}^{t_2}\norm{\mathcal{T}(t_2-s)}\norm{A_{1}}\norm{S_{n}(s-\tau,n\tau)x}\mathrm{d}s, \quad x\in X.
	\end{align}	
	As $t_1 \to t_2$, the integrand in the first term on the right of \eqref{ineq} converges to zero boundedly, the integrand of the second term is bounded and the second integral converges to zero boundedly, too. It follows that $S_{n+1}(t,(n+1)\tau)\mathds{1}_{t\geq (n+1)\tau}$ is strongly continuous on $[0,\infty)$. Therefore, with the help of mathematical induction principle, we have showed that $S_{n}(t,n\tau)\mathds{1}_{t\geq n\tau}$ is well-defined, strongly continuous, and satisfies \eqref{exponential} for all $n \in \mathbb{N}_{0}$. Hence, the series \eqref{power} is a strongly continuous function on $[0,\infty)$ with values in $\mathcal{L}(X)$.
	
	\textbf{2.} Meanwhile, by making use of the estimation \eqref{exponential} and closed-form of a \textit{pure delayed} exponential function, in accordance with the \textit{comparison test} for functional series, the series representing in \eqref{power} converges in the \textit{uniform operator topology uniformly} for $t$ in every compact subset of $[0,\infty)$:
	\begin{align*}
		\norm{\sum\limits_{n=0}^{\infty}S_{n}(t,n\tau)\mathds{1}_{t\geq n\tau}}&\leq \sum\limits_{n=0}^{\infty}\norm{S_{n}(t,n\tau)\mathds{1}_{t\geq n\tau}}\\
		&\leq M\exp(\omega t)\sum\limits_{n=0}^{\infty}\Big( M\norm{A_{1}}\exp(-\omega \tau)\Big)^{n}\frac{(t-n\tau)^{n}}{n!}\mathds{1}_{t\geq n\tau}\\
		&=M\exp(\omega t)\exp_{\tau}(\hat{\omega}t), \quad t\geq 0, \quad \hat{\omega}\coloneqq M\norm{A_{1}}\exp(-\omega \tau),
	\end{align*}
	where $\exp_{\tau}(\cdot):\mathbb{R}\to \mathbb{R}$ is a \textit{pure delayed} real-valued exponential function defined by
	\begin{equation*}
		\exp_{\tau}(t)=\sum\limits_{n=0}^{\infty}\frac{(t-n\tau)^{n}}{n!}\mathds{1}_{t\geq n\tau}, \quad t\in \mathbb{R}.
	\end{equation*}
	On the other hand, the series \eqref{power} is majorized by the series expansion of $M\exp(\omega_{1}t)$ where $\omega_{1}\coloneqq \omega +\hat{\omega}$ with $\hat{\omega}\coloneqq M\norm{A_{1}}\exp(-\omega \tau)$
	\begin{align*}
		\norm{\sum\limits_{n=0}^{\infty}S_{n}(t,n\tau)\mathds{1}_{t\geq n\tau}}
		&\leq M\exp(\omega t)\sum\limits_{n=0}^{\infty}\hat{\omega}^{n}\frac{(t-n\tau)^{n}}{n!}\mathds{1}_{t\geq n\tau}\\
		&\leq M\exp(\omega t)\sum\limits_{n=0}^{\infty}\hat{\omega}^{n}\frac{t^{n}}{n!}\\
		&=M\exp\Big(\Big(\omega+\hat{\omega}\Big)t\Big)\\
		&\coloneqq M \exp(\omega_{1}t), \quad t\geq 0.
	\end{align*}
	Then, for $Re(\lambda_{0})>\omega_1$, we can attain that	
	\begin{align*}
		\int_{0}^{\infty}\exp(-\lambda_{0}t)\sum_{n=0}^{\infty}\Big[S_{n}(t,n\tau)\mathds{1}_{t\geq n\tau}x\Big]\mathrm{d}t
		=&\sum_{n=0}^{\infty}\int_{0}^{\infty}\exp(-\lambda_{0}t)S_{n}(t,n\tau)\mathds{1}_{t\geq n\tau}x\mathrm{d}t\\=&\sum\limits_{n=0}^{\infty}\int_{n\tau}^{\infty}\exp(-\lambda_{0}t)S_{n}(t,n\tau)x\mathrm{d}t, \quad t\geq n\tau, \quad x\in X,
	\end{align*}	
	where the interchanging of the summation and integration is justified by the \textit{uniform convergence} of the series in the \textit{uniform operator topology}.
	
	Now, if $x^{*}\in X^{*}$, it is a consequence of the uniform convergence of the integral and of the Fubini's theorem that
	\allowdisplaybreaks
	\begin{align*}
		&x^{*}\Big[\int_{n\tau}^{\infty}\exp(-\lambda_{0}t)S_{n}(t,n\tau)x\mathrm{d}t\Big]\\
		&=\int_{n\tau}^{\infty}\exp(-\lambda_{0}t)x^{*}\Big[S_{n}(t,n\tau)x\Big]\mathrm{d}t\\
		&=\int_{n\tau}^{\infty}\exp(-\lambda_{0}t)\int_{n\tau}^{t}x^{*}\Big[\mathcal{T}(t-s)A_{1}S_{n-1}(s-\tau,(n-1)\tau)x\Big]\mathrm{d}s\mathrm{d}t\\
		&=\int_{n\tau}^{\infty}\exp(-\lambda_{0}s)\int_{s}^{\infty}\exp(-\lambda_{0}(t-s))x^{*}\Big[\mathcal{T}(t-s)A_{1}S_{n-1}(s-\tau,(n-1)\tau)x\Big]\mathrm{d}t\mathrm{d}s\\
		&=\int_{n\tau}^{\infty}\exp(-\lambda_{0}s)\int_{0}^{\infty}\exp(-\lambda_{0}t)x^{*}\Big[\mathcal{T}(t)A_{1}S_{n-1}(s-\tau,(n-1)\tau)x\Big]\mathrm{d}t\mathrm{d}s\\
		&=\int_{n\tau}^{\infty}\exp(-\lambda_{0}s)x^{*}\Big[\mathcal{R}\left( \lambda_{0};A_{0}\right)A_{1} S_{n-1}(s-\tau,(n-1)\tau)x\Big]\mathrm{d}s\\
		&=x^{*}\Big[\mathcal{R}\left( \lambda_{0};A_{0}\right)A_{1}\left\lbrace \int_{n\tau}^{\infty}\exp(-\lambda_{0}s)S_{n-1}(s-\tau,(n-1)\tau)x\mathrm{d}s\right\rbrace \Big], \quad t\geq n\tau,\quad x\in X.
	\end{align*}
	Hence, by induction, we  derive that
	\begin{equation*}
		\int_{n\tau}^{\infty}\exp(-\lambda_{0}t)S_{n}(t,n\tau)x\mathrm{d}t=\mathcal{R}\left( \lambda_{0};A_{0}\right)\Big[ A_{1}\mathcal{R}\left( \lambda_{0};A_{0}\right)\Big]^{n}\exp(-n\lambda_{0}\tau)x, \quad t\geq n\tau, \quad x\in X.
	\end{equation*}
	Therefore, the Laplace transform of the series \eqref{power} for any $t\in[0,\infty)$ is defined by
	\begin{equation*}
		\int_{0}^{\infty}\exp(-\lambda_{0}t)\Big[\sum_{n=0}^{\infty}S_{n}(t,n\tau)\mathds{1}_{t\geq n\tau}x\Big]\mathrm{d}t=\sum_{n=0}^{\infty}\mathcal{R}\left( \lambda_{0};A_{0}\right)\Big[ A_{1}\mathcal{R}\left( \lambda_{0};A_{0}\right)\Big]^{n}\exp(-n\lambda_{0}\tau)x, \quad x\in X.
	\end{equation*}
	On the other hand, for  $Re(\lambda_{0})>\omega_{1}=\omega+M\norm{A_{1}}\exp\left( -\omega \tau\right)$, we have
	\begin{align}\label{omega-1}
		\|A_{1}\mathcal{R}\left(\lambda_{0};A_{0}\right)\exp(-\lambda_{0}\tau)\|
		&=\norm{A_{1}}\norm{\int_{0}^{\infty}\exp(-\lambda_{0}(t+\tau))\mathcal{T}(t)\mathrm{d}t}\nonumber\\
		&=\norm{A_{1}}\norm{\int_{\tau}^{\infty}\exp(-\lambda_{0}t)\mathcal{T}(t-\tau)\mathrm{d}t}\nonumber\\
		&\leq \norm{A_{1}}\int_{0}^{\infty}\exp\Big(-Re(\lambda_{0})t\Big)\norm{\mathcal{T}(t-\tau)}\mathrm{d}t\nonumber\\
		&\leq M \norm{A_{1}}\exp(-\omega \tau)\int_{0}^{\infty}\exp\Big(-\left( Re(\lambda_{0})-\omega\right) t\Big)\mathrm{d}t\nonumber\\
		&=\frac{M \norm{A_{1}}\exp(-\omega \tau)}{Re(\lambda_{0})-\omega}<1.
	\end{align} 
	Therefore, for the infinitesimal generator $A_{0}$ of a strongly continuous semigroup of bounded linear operators $\left\lbrace \mathcal{T}(t)\right\rbrace_{t\geq 0}$ and the bounded linear operator $A_{1}$, by Lemma \ref{lema}, we have:
	\begin{align*}
		\mathcal{R}\left( \lambda_{0};A_{0}+\hat{A}_{1}\right)x&=\int_{0}^{\infty}\exp(-\lambda_{0}t)\mathcal{S}(t;\tau)x\mathrm{d}t\\
		&=\sum_{n=0}^{\infty}\mathcal{R}\left( \lambda_{0};A_{0}\right)\Big[ A_{1}\mathcal{R}\left( \lambda_{0};A_{0}\right)\Big]^{n}
		\exp(-n\lambda_{0}\tau)x, \quad x \in X.
	\end{align*}
	Therefore, for $Re(\lambda_{0})>\omega_{1}$, the Laplace integral transforms of both $\mathcal{S}(t;\tau)$ and $\sum_{n=0}^{\infty}
	S_{n}(t,n\tau)\mathds{1}_{t\geq n\tau}$ are equal on $[0,\infty)$ and hence by the uniqueness theorem of Laplace transform (\cite{engel-nagel}, pp. 530), these two functions are equal for any $t\in [0,\infty)$.
	Therefore, $\sum_{n=0}^{\infty}
	S_{n}(t,n\tau)\mathds{1}_{t\geq n\tau}$ converges to a strongly continuous function $\mathcal{S}(t;\tau)$ uniformly with respect to $t$ in the uniform operator topology on every compact subsets of $[0,\infty)$ and $\mathcal{S}(t;\tau)$ is a fundamental solution of \eqref{main equation} for $t\geq 0$, under the initial conditions \eqref{unit}.
	
	$\textbf{3.}$ Alternatively, we can prove that  $t \mapsto \mathcal{S}(t;\tau)$ satisfies \eqref{main equation} for all $t\geq 0$ with initial conditions $\mathcal{S}(t;\tau)=\Theta$, $-\tau \leq t<0$ and $\mathcal{S}(0;\tau)=I$. Since $\mathcal{T}(0)=I$, $S_{n}(0,n\tau)=\Theta$ for $n \in \mathbb{N}$, we have $\mathcal{S}(t;\tau)=\Theta$, $-\tau\leq t<0$ and $\mathcal{S}(0;\tau)=I$, i.e., the initial conditions \eqref{unit} are satisfied. Applying \eqref{power}, \eqref{recursive} and interchanging of the summation and integration which is justified by the uniform convergence of the series in the uniform operator topology, it follows that
	\allowdisplaybreaks
	\begin{align*}
		\mathcal{S}(t;\tau)&=\sum_{n=0}^{\infty}S_{n}(t,n\tau)\mathds{1}_{t\geq n\tau}
		=S_{0}(t,0)+\sum_{n=1}^{\infty}S_{n}(t,n\tau)\mathds{1}_{t\geq n\tau}\\
		&=\mathcal{T}(t)+\sum_{n=1}^{\infty}\int_{0}^{t}\mathcal{T}(t-s)A_{1}S_{n-1}(s-\tau,(n-1)\tau)\mathds{1}_{s\geq n\tau}\mathrm{d}s \\
		&=\mathcal{T}(t)+\int_{0}^{t}\mathcal{T}(t-s)A_{1}\sum_{n=1}^{\infty}S_{n-1}(s-\tau,(n-1)\tau)\mathds{1}_{s \geq n\tau}\mathrm{d}s \\
		&=\mathcal{T}(t)+\int_{0}^{t}\mathcal{T}(t-s)A_{1}\sum_{n=0}^{\infty}S_{n}(s-\tau,n\tau)\mathds{1}_{s\geq (n+1)\tau}\mathrm{d}s\\
		&=\mathcal{T}(t)+\int_{0}^{t}\mathcal{T}(t-s)A_{1}\mathcal{S}(s-\tau;\tau)\mathrm{d}s, \quad t\geq 0.
	\end{align*}
	First, we need to show that the function $[0,\infty) \ni t\mapsto \mathcal{S}(t,\tau) \in \mathcal{L}\left(X\right)$ is strongly continuously differentiable. Since $t\mapsto \mathcal{T}(t)$ is strongly continuously differentiable for all $t\geq0$, this implies that  $t\mapsto S_{0}(t,0)$ is strongly continuously differentiable for all $t\geq 0$. Assuming that this is true for $t \mapsto S_{n}(t,n\tau)$ for all $t\geq 0$. Then, by induction, it is easily shown to be true for $t \mapsto S_{n+1}(t,(n+1)\tau)$ for all $t\geq 0$.
	Furthermore, by Leibniz integral rule, we derive that 
	\begin{align*}
		&S_{0}^{\prime}(t,0)=A_{0}\mathcal{T}(t)=\mathcal{T}(t)A_{0}, \quad t\geq 0\\
		&S_{n}^{\prime}(t,n\tau)=\int_{n\tau}^{t}\mathcal{T}(t-s)A_{1}S_{n-1}^{\prime}(s-\tau,(n-1)\tau)\mathrm{d}s, \quad t\geq n\tau, \quad n \in \mathbb{N}.
	\end{align*}
	\allowdisplaybreaks
	By using \eqref{exponentiall}, for $x\in \mathcal{D}(A_{0})$, we obtain by induction from these equations the following bound:
	\begin{equation*}
		\norm{S_{n}^{\prime}(t,n\tau)x}\leq M\Big(M\norm{A_{1}}\exp(-\omega \tau)\Big)^{n}\exp(\omega t)\frac{(t-n\tau)^{n}}{n!}\norm{A_{0}x},\quad t\geq n\tau,\quad  n\in\mathbb{N}_{0}.
	\end{equation*}
	From this bound it follows that the series $\sum_{n=0}^{\infty}S_{n}^{\prime}(t,n\tau)\mathds{1}_{t\geq n\tau}$ converges uniformly in every compact interval of $[0,\infty)$ to continuous function which, by the usual argument, is $\mathcal{S}^{\prime}(t;\tau)$.
	
	To show $\mathcal{S}(t;\tau)$  is a fundamental solution of \eqref{main equation} for $t\geq 0$ under initial conditions \eqref{unit}, we differentiate last expression with the help of Leibniz integral rule (differentiation under integral sign), as follows:
	\begin{align*}
		\frac{\mathrm{d}}{\mathrm{d}t}\mathcal{S}(t;\tau)&=
		\frac{\mathrm{d}}{\mathrm{d}t}\Big[\mathcal{T}(t)+\int_{0}^{t}\mathcal{T}(t-s)A_{1}\mathcal{S}(s-\tau)\mathrm{d}s\Big]\\
		&=A_{0}\Big[\mathcal{T}(t)+\int_{0}^{t}\mathcal{T}(t-s)A_{1}\mathcal{S}(s-\tau;\tau)\mathrm{d}s\Big]+A_{1}\mathcal{S}(t-\tau;\tau)\\
		&=A_{0}\mathcal{S}(t;\tau)+A_{1}\mathcal{S}(t-\tau;\tau), \quad t\geq 0.
	\end{align*}
	
	$\textbf{4.}$ In the uniqueness proof, it will be sufficient to show that if $\mathcal{U}(t)$ solves the abstract functional differential equation \eqref{main equation} for $t\geq 0$ with zero initial conditions $\mathcal{U}(t)=0$ for $t\in[-\tau,0]$, then $\mathcal{U}(t)\equiv0$ for all $t\in [0,\infty)$, since $A_{0}$ is densely-defined in $X$.
	
	In other words, it is sufficient to consider a strongly continuously differentiable function $\mathcal{U}(t)$ on $[0,\infty)$ to $\mathcal{L}\left(X\right)$ such that $\mathcal{U}(t)=0$ for $-\tau\leq t\leq 0$ and $\frac{\mathrm{d}}{\mathrm{d}t}\mathcal{U}(t)=A_{0}\mathcal{U}(t)+A_{1}\mathcal{U}(t-\tau)$ for $t\geq 0$.
	Operating on both sides of this equation by $\mathcal{T}(t-s)$ and integrating on $[0,t]$ gives
	\begin{align}\label{int-1}
		\int_{0}^{t}\mathcal{T}(t-s)\frac{\partial}{\partial s}\mathcal{U}(s)\mathrm{d}s&=
		\int_{0}^{t}\mathcal{T}(t-s)A_{0}\mathcal{U}(s)\mathrm{d}s\nonumber\\&+
		\int_{0}^{t}\mathcal{T}(t-s)A_{1}\mathcal{U}(s-\tau)\mathrm{d}s, \quad t\geq 0.
	\end{align}
	It can be easily shown for $s\in[0,t]$ that
	\begin{align}\label{int-2}
		\frac{\partial}{\partial s}\Big[\mathcal{T}(t-s)\mathcal{U}(s)]=-\mathcal{T}(t-s)A_{0}\mathcal{U}(s)+\mathcal{T}(t-s)\frac{\partial}{\partial s}\mathcal{U}(s).
	\end{align}
	Therefore, by virtue of  \eqref{int-1}, \eqref{int-2} and the \textit{second fundamental theorem of calculus}, we attain that
	\begin{align}\label{last int}
		\mathcal{U}(t)=\int_{0}^{t}\mathcal{T}(t-s)A_{1}\mathcal{U}(s-\tau)\mathrm{d}s, \quad t\geq 0.
	\end{align}
	Let  $\mathcal{V}(t)\coloneqq \sup\left\lbrace \mathcal{U}(t+h): h\in [-\tau,0]\right\rbrace$. For a fixed time-delay $\tau>0$,  we are setting $m_{t}=\sup\left\lbrace \norm{\mathcal{V}(s)}: s \in [0,t]\right\rbrace $ and we see that 
	\begin{align*}
		m_{t}\leq M\norm{A_{1}}  \frac{e^{\omega t}-1}{\omega}m_{t},
	\end{align*}
	and for chosen sufficiently small $t$ such that $M\norm{A_{1}}  \frac{e^{\omega t}-1}{\omega}<1$. This implies that $m_{t}=0$. Thus, $\mathcal{U}(t)=0$ on $[0,t_{0}]$ with $t_{0}>0$. Iteration of this argument leads to $\mathcal{U}(t)\equiv0$ on $[0,\infty)$.
	
	Assume that a strongly continuously differentiable function  $\mathcal{C}(t;\tau)$ on $[0,\infty)$ to $\mathcal{L}\left(X\right)$ such that $\mathcal{C}(t;\tau)=\Theta$ for $-\tau\leq t<0$, $\mathcal{C}(0;\tau)=I$ and
	\begin{equation}\label{c}
		\frac{\mathrm{d}}{\mathrm{d}t}\mathcal{C}(t;\tau)=A_{0}\mathcal{C}(t;\tau)+A_{1}\mathcal{C}(t-\tau;\tau), \quad t\geq 0.
	\end{equation}
	For operating on both sides of \eqref{c} by $\mathcal{T}(t-s)$ and integrating on $[0,t]$ gives for any $t\geq$:
	\begin{align}\label{int-3}
		\int_{0}^{t}\mathcal{T}(t-s)\frac{\partial}{\partial s}\mathcal{C}(s;\tau)\mathrm{d}s&=
		\int_{0}^{t}\mathcal{T}(t-s)A_{0}\mathcal{C}(s;\tau)\mathrm{d}s\nonumber\\
		&+
		\int_{0}^{t}\mathcal{T}(t-s)A_{1}\mathcal{C}(s-\tau;\tau)\mathrm{d}s. 
	\end{align}
	It can be easily shown for $s\in[0,t]$ that
	\begin{align}\label{int-4}
		\frac{\partial}{\partial s}\Big[\mathcal{T}(t-s)\mathcal{C}(s;\tau)]=-\mathcal{T}(t-s)A_{0}\mathcal{C}(s;\tau)+\mathcal{T}(t-s)\frac{\partial}{\partial s}\mathcal{C}(s;\tau).
	\end{align}
	Therefore, from \eqref{int-3} and \eqref{int-4} and by the \textit{second theorem of fundamental calculus}, we obtain that
	\begin{align}
		\mathcal{C}(t;\tau)=\mathcal{T}(t)+\int_{0}^{t}\mathcal{T}(t-s)A_{1}\mathcal{C}(s-\tau;\tau)\mathrm{d}s, \quad t\geq 0.
	\end{align}
	On the other hand, the method of \textit{variation of constant formula} yields $\mathcal{S}(t;\tau)$ on $[0,\infty)$ to $\mathcal{L}\left(X\right)$ is also satisfying
	\begin{align}
		\mathcal{S}(t;\tau)=\mathcal{T}(t)+\int_{0}^{t}\mathcal{T}(t-s)A_{1}\mathcal{S}(s-\tau;\tau)\mathrm{d}s, \quad t\geq 0.
	\end{align}
	The difference $\mathcal{U}(t)=\mathcal{S}(t;\tau)-\mathcal{C}(t;\tau)$, $t\geq 0$ satisfies \eqref{last int} and vanishes at the points $t\in [-\tau,0]$. Thus, the uniqueness argument shows that this difference is identically zero for any $t \in [0,\infty)$.
	The proof is complete.
\end{proof}

\begin{remark}
	Note that the solution has can also represented by 
	\begin{equation*}
		\mathcal{S}(t;\tau)=\sum_{n=0}^{\infty}S_{n}(t,n\tau)\mathds{1}_{t\geq n\tau}, \quad t\geq 0,
	\end{equation*}
	via the following successive iterations:
	\begin{align*}
		&S_{0}(t,0)=\mathcal{T}(t),\quad t\geq 0,\nonumber\\
		&S_{n}(t,n\tau)=\int_{0}^{t-n\tau}S_{n-1}(t-s-\tau,(n-1)\tau)A_{1}\mathcal{T}(s)\mathrm{d}s, \quad t\geq n\tau, \quad n \in \mathbb{N}.
	\end{align*}
\end{remark}

If we consider abstract differential equation with a discrete delay \eqref{main equation} on $[-\tau,T]$ where $T=(n+1)\tau$ for a fixed $n\in \mathbb{N}_{0}$, then we can introduce a \textit{piece-wise} construction for a \textit{fundamental solution} $\mathcal{S}\left( \cdot;\tau\right):[-\tau,T] \to \mathcal{L}\left(X\right)$ of \eqref{main equation} under the initial conditions \eqref{unit}.

\begin{corollary}
	Let $A_{0}:\mathcal{D}\left( A_{0}\right) \subseteq X\to X$ be the infinitesimal generator of a strongly continuous semigroup of bounded linear operators $\mathcal{T}(t)$, $0\leq t\leq T$ and $A_{1} \in \mathcal{L}(X)$. Then, the abstract Cauchy problem for functional evolution equation \eqref{main equation} admits an uniquely determined strongly continuous fundamental solution $\mathcal{S}(\cdot;\tau):[-\tau,T] \to \mathcal{L}\left(X\right)$ on $[-\tau,T]$ which is satisfying  initial conditions $\mathcal{S}(t;\tau)=\Theta$, $-\tau\leq t<0$ and $\mathcal{S}(0;\tau)=I$ and strongly continuously differentiable on $[0,T]$ and
	\begin{equation}\label{son-1}
		\mathcal{S}(t;\tau)=\sum_{k=0}^{n}S_{k}(t,k\tau), \quad n\tau<t\leq  (n+1)\tau, \quad  n\in \mathbb{N}_{0},
	\end{equation}
	where 
	\begin{align*}
		S_{0}(t,0)&=\mathcal{T}(t),\quad t\geq 0, \nonumber\\
		S_{k}(t,k\tau)&=\int_{k\tau}^{t}\mathcal{T}(t-s)A_{1}S_{k-1}(s-\tau,(k-1)\tau)\mathrm{d}s\nonumber\\
		&=\int_{0}^{t-k\tau}S_{k-1}(t-s-\tau,(k-1)\tau)A_{1}\mathcal{T}(s)\mathrm{d}s, \quad t\geq k\tau, \quad  k=1,2,\ldots,n.
	\end{align*}
\end{corollary}

In the particular case, by using the \textit{powers} of $\mathcal{R}(\lambda_{0};A_{0})$, we can derive the following elegant representation formula for a fundamental solution $\mathcal{S}(t;\tau)$, $t\geq -\tau$ of \eqref{main equation} under the initial conditions \eqref{unit}.

The following theorem is given in more general case, for $C_{0}$-groups. If $C_{0}$-group $\left\lbrace \mathcal{T}(t)\right\rbrace_{t \in \mathbb{R}}$ commutes with $A_{1}\in \mathcal{L}(X)$, one can obtain more elegant formula for a fundamental solution as below.
\begin{theorem}\label{permutable-case}
	Let $A_{0}:\mathcal{D}\left( A_{0}\right) \subseteq X\to X$ be the infinitesimal generator of a strongly continuous group $\left\lbrace \mathcal{T}(t)\right\rbrace_{t \in \mathbb{R}}$ of bounded linear operators. If a strongly continuous group $\left\lbrace \mathcal{T}(t)\right\rbrace_{t \in \mathbb{R}}$ commutes with $A_{1}\in \mathcal{L}(X)$, then, the strongly continuous one-parameter family of bounded linear operators $\mathcal{S}(t;\tau)$ which is satisfying $\mathcal{S}(t;\tau)=\Theta$ for $-\tau \leq t <0$ and $\mathcal{S}(0;\tau)=I$, has a closed-form on $[0,\infty)$ as follows:
	\begin{align}\label{commute}
		\mathcal{S}(t;\tau)&=\exp_{\tau}^{[A_{1}\mathcal{T}(\tau)^{-1}]t}\mathcal{T}(t)=\sum\limits_{n=0}^{\infty}\Big[A_{1}\mathcal{T}(\tau)^{-1}\Big]^{n}\frac{(t-n\tau)^{n}}{n!}\mathds{1}_{t\geq n\tau}\mathcal{T}(t)\nonumber\\
		&=\mathcal{T}(t)\exp_{\tau}^{[A_{1}\mathcal{T}(\tau)^{-1}]t}=\mathcal{T}(t)\sum\limits_{n=0}^{\infty}\Big[A_{1}\mathcal{T}(\tau)^{-1}\Big]^{n}\frac{(t-n\tau)^{n}}{n!}\mathds{1}_{t\geq n\tau}, \quad t\geq0.
	\end{align}
\end{theorem}

\begin{proof}
	Since a $C_{0}$-group $\mathcal{T}(\cdot):\mathbb{R} \to \mathcal{L}(X)$ commutes with $A_{1}\in \mathcal{L}\left(X\right)$, then by the virtue of Lemma \ref{lema} and relation \eqref{omega-1}, for $Re(\lambda_{0})>\omega_{1}=\omega+M\norm{A_{1}}\exp\left(-\omega \tau\right)$ the following identity holds true:
	\begin{align}\label{1}
		\mathcal{R}(\lambda_{0};A_{0}+\hat{A}_{1})&= \int_{0}^{\infty}\exp(-\lambda_{0} t)\mathcal{S}(t;\tau)\mathrm{d}t\nonumber\\&=\sum_{n=0}^{\infty}\mathcal{R}(\lambda_{0};A_{0})\Big[A_{1}\mathcal{R}(\lambda_{0};A_{0})\exp(-\lambda_{0}\tau)\Big]^{n}\nonumber\\
		&=\sum_{n=0}^{\infty}A_{1}^{n}\Big[\mathcal{R}(\lambda_{0};A_{0})\Big]^{n+1}\exp(-n\lambda_{0} \tau), \quad \hat{A}_{1}= A_{1}\exp(-\lambda_{0}\tau)\in \mathcal{L}(X).
	\end{align}
	It is known that [\cite{engel-nagel}, Corollary 1.11, pp. 56], for $Re(\lambda_{0})>\omega$ and $n\in \mathbb{N}_{0}$, the following identity is true:
	\begin{equation}\label{2}
		\Big[\mathcal{R}(\lambda_{0};A_{0})\Big]^{n+1}x=\int_{0}^{\infty}\exp(-n\lambda_{0}t)\frac{t^{n}}{n!}\mathcal{T}(t)x\mathrm{d}t, \quad x \in X.
	\end{equation}
	By the identity \eqref{2} and using integration by substitution, we get 
	\begin{equation}\label{3}
		\Big[\mathcal{R}(\lambda_{0};A_{0})\Big]^{n+1}\exp(-n\lambda_{0}\tau)x=\int_{0}^{\infty}\exp(-\lambda_{0} t)\frac{(t-n\tau)^{n}}{n!}\mathcal{T}(t-n\tau)\mathds{1}_{t\geq n\tau}x\mathrm{d}t, \quad x \in X.
	\end{equation}
	Therefore, from \eqref{1} and \eqref{3}, taking inverse Laplace transform, we derive a desired result:
	\begin{align*}
		\mathcal{S}(t;\tau)&=\sum\limits_{n=0}^{\infty}A_{1}^{n}\frac{(t-n\tau)^{n}}{n!}\mathcal{T}(t-n\tau)\mathds{1}_{t\geq n\tau}\nonumber\\
		&=\sum\limits_{n=0}^{\infty}\Big[A_{1}\mathcal{T}(\tau)^{-1}\Big]^{n}\frac{(t-n\tau)^{n}}{n!}\mathds{1}_{t\geq n\tau}\mathcal{T}(t)\nonumber\\
		&=\mathcal{T}(t)\sum\limits_{n=0}^{\infty}\Big[A_{1}\mathcal{T}(\tau)^{-1}\Big]^{n}\frac{(t-n\tau)^{n}}{n!}\mathds{1}_{t\geq n\tau}\nonumber\\
		&=\exp_{\tau}^{[A_{1}\mathcal{T}(\tau)^{-1}]t}\mathcal{T}(t)\\
		&=\mathcal{T}(t)\exp_{\tau}^{[A_{1}\mathcal{T}(\tau)^{-1}]t}, \quad t\geq 0,
	\end{align*}
	where by the properties of $C_{0}$-group $\left\lbrace \mathcal{T}(t)\right\rbrace _{t \in \mathbb{R}}$ we have used that
	\begin{equation*}
		\mathcal{T}(t-n\tau)=\mathcal{T}(t)\mathcal{T}(-n\tau)=\mathcal{T}(t)\Big[\mathcal{T}(-\tau)\Big]^{n}=\mathcal{T}(t)\Big[\mathcal{T}(\tau)^{-1}\Big]^{n}.
	\end{equation*}
	The proof is complete.
\end{proof}
Furthermore, the following corollary deals with the special case where the $C_{0}$-semigroup and $A_{1}\in \mathcal{L}\left(X\right)$ are permutable.
\begin{corollary}
	Let $A_{0}:\mathcal{D}\left( A_{0}\right) \subseteq X\to X$ be the infinitesimal generator of a strongly continuous semigroup $\left\lbrace \mathcal{T}(t)\right\rbrace_{t \geq 0}$ of bounded linear operators. If a strongly continuous group $\left\lbrace \mathcal{T}(t)\right\rbrace_{t\geq 0}$ commutes with $A_{1}\in \mathcal{L}(X)$, then, the strongly continuous one-parameter family of bounded linear operators $\mathcal{S}(t;\tau)$ which is satisfying $\mathcal{S}(t;\tau)=\Theta$ for $-\tau \leq t <0$ and $\mathcal{S}(0;\tau)=I$, has a closed-form on $[0,\infty)$ as follows:
	\begin{align}\label{partially}
		\mathcal{S}(t;\tau)&=\sum\limits_{n=0}^{\infty}A_{1}^{n}\frac{(t-n\tau)^{n}}{n!}\mathcal{T}(t-n\tau)\mathds{1}_{t\geq n\tau}\nonumber\\
		&=\sum\limits_{n=0}^{\infty}\mathcal{T}(t-n\tau)\frac{(t-n\tau)^{n}}{n!}\mathds{1}_{t\geq n\tau}A_{1}^{n}, \quad t\geq 0.
	\end{align}
\end{corollary}

If we consider this particular case on $[-\tau,T]$ where $T=(n+1)\tau$ for a fixed $n\in \mathbb{N}_{0}$, then we can use the following \textit{piece-wise} construction for a \textit{fundamental solution} of \eqref{main equation} with the initial conditions \eqref{unit}.
\begin{corollary}\label{cor}
	Let $A_{0}:\mathcal{D}\left( A_{0}\right) \subseteq X\to X$ be the infinitesimal generator of a strongly continuous group $\left\lbrace \mathcal{T}(t)\right\rbrace_{t \in \mathbb{R}}$ of bounded linear operators. If a strongly continuous group $\left\lbrace \mathcal{T}(t)\right\rbrace_{t \in \mathbb{R}}$ commutes with $A_{1}\in \mathcal{L}(X)$, then, the fundamental solution of \eqref{main equation} has a representation on $[0,T]$ as follows:
	\begin{align*}
		\mathcal{S}(t;\tau)&=\exp_{\tau}^{[A_{1}\mathcal{T}(\tau)^{-1}]t}\mathcal{T}(t)=\sum\limits_{k=0}^{n}\Big[A_{1}\mathcal{T}(\tau)^{-1}\Big]^{k}\frac{(t-k\tau)^{k}}{k!}\mathcal{T}(t)\\
		&=\mathcal{T}(t)\exp_{\tau}^{[A_{1}\mathcal{T}(\tau)^{-1}]t}=\mathcal{T}(t)\sum\limits_{k=0}^{n}\Big[A_{1}\mathcal{T}(\tau)^{-1}\Big]^{k}\frac{(t-k\tau)^{k}}{k!}, \quad n\tau< t\leq  (n+1)\tau.
	\end{align*}
\end{corollary}

Next, we derive an explicit representation formula is known as the \textit{Cauchy formula} for the classical (strong) solution of the abstract initial value problem to linear homogeneous functional evolution equation \eqref{main equation} via the method of \textit{variation of constants formula}.

\begin{theorem}\label{representation}
	Let $A_{0}:\mathcal{D}\left( A_{0}\right) \subseteq X\to X$ be infinitesimal generator of a strongly continuous semigroup $\left\lbrace \mathcal{T}(t)\right\rbrace _{t \geq 0}$ of bounded linear operators, $A_{1} \in \mathcal{L}(X)$ and the initial function $\varphi(\cdot)\in \mathbb{C}^{1}\left( [-\tau,0],X\right)$. Then, the classical (strong) solution $u_{0}(\cdot)\in \mathbb{C}^{1}\left( [-\tau,\infty),X\right)$ of the abstract Cauchy problem \eqref{main equation} for a linear homogeneous functional evolution equation is satisfying $u_{0}(t)\in \mathcal{D}\left( A_{0}\right)$ for all $t\geq 0$ with $\varphi(t)\in \mathcal{D}\left( A_{0}\right)$ for all $t\in [-\tau,0]$ and it can be represented in the integral form
	\begin{equation}\label{sol-1}
		u_{0}(t)=\mathcal{S}(t+\tau;\tau)\varphi(-\tau)+
		\int_{-\tau}^{0}\mathcal{S}(t-s;\tau)\Big[ \varphi^{\prime}(s)-A_{0}\varphi(s)\Big]\mathrm{d}s, \quad t \geq -\tau.
	\end{equation}
\end{theorem}

\begin{proof}
	By using the \textit{variation of constants formula}, any solution $u_{0}(\cdot):[-\tau,\infty) \to X$ of linear homogeneous delay evolution equation \eqref{main equation} should be satisfied in the form:
	\begin{equation*}
		u_{0}(t)=\mathcal{S}(t+\tau;\tau)c+
		\int_{-\tau}^{0}\mathcal{S}(t-s;\tau)g(s) \mathrm{d}s, \quad t \geq -\tau,
	\end{equation*}
	where $c$ is an unknown vector, $g(\cdot):[-\tau,0]\to X$ is an unknown continuously differentiable operator-valued function and that it satisfies the initial conditions $u_{0}(t)=\varphi(t)$ for $-\tau \leq t\leq 0$:
	\begin{equation}\label{int}
		\varphi(t)=\mathcal{S}(t+\tau;\tau)c+
		\int_{-\tau}^{0}\mathcal{S}(t-s;\tau)g(s) \mathrm{d}s, \quad  t \in [-\tau,0].
	\end{equation}
	If $t=-\tau$, then in accordance with the following result
	\begin{equation*}
		\mathcal{S}(-\tau-s;\tau)=\begin{cases}
			\Theta, \quad -\tau <s \leq 0,\\
			I, \qquad s=-\tau,
		\end{cases}
	\end{equation*}
	we acquire $c=\varphi(-\tau)\in \mathcal{D}(A_{0})$.
	
	On the interval $-\tau \leq t \leq 0$, we split \eqref{int} into two integrals:
	\begin{align*}
		\varphi(t)=\mathcal{S}(t+\tau;\tau)\varphi(-\tau)&+
		\int_{-\tau}^{t}\mathcal{S}(t-s;\tau)g(s)\mathrm{d}s\\ &+\int_{t}^{0}\mathcal{S}(t-s;\tau)g(s) \mathrm{d}s, \quad t \in [-\tau,0].
	\end{align*}
	Furthermore, by making use of the following relations:
	\begin{equation*}
		\mathcal{S}(t-s;\tau)=\mathcal{T}(t-s), \quad -\tau \leq s \leq t \quad \text{and} \quad 
		\mathcal{S}(t-s;\tau)=\begin{cases}
			\Theta, \quad t <s \leq 0,\\
			I, \qquad s=t,
		\end{cases}
	\end{equation*}
	one can easily derive that
	\begin{equation}\label{difer}
		\varphi(t)=\mathcal{T}(t+\tau)\varphi(-\tau)+
		\int_{-\tau}^{t}\mathcal{T}(t-s)g(s)\mathrm{d}s, \quad  t\in [-\tau,0]. 
	\end{equation}
	After differentiating the equation \eqref{difer} by \textit{Leibniz integral rule}, we derive
	\begin{align*}
		\varphi^{\prime}(t)&=A_{0}\Big[\mathcal{T}(t+\tau)\varphi(-\tau)+
		\int_{-\tau}^{t}\mathcal{T}(t-s)g(s)\mathrm{d}s\Big]+g(t)\\
		&=A_{0}\varphi(t)+g(t), \quad t\in [-\tau,0].
	\end{align*}
	Hence, it follows that $g(t)=\varphi^{\prime}(t)-A_{0}\varphi(t)$, $t\in [-\tau,0]$.
	The proof is complete.
\end{proof}

The integrand occurring in the expression \eqref{sol-1} is the derivative of the initial function $\varphi(t)$, $-\tau \leq t\leq 0$, i.e., the initial conditions are continuously differentiable on the initial interval $[-\tau,0]$. This condition can be "weakened" by requiring  $\varphi(t)$ only to be continuous for $t\in [-\tau,0]$.

\begin{theorem}
	Let $A_{0}:\mathcal{D}\left( A_{0}\right) \subseteq X\to X$ be infinitesimal generator of a strongly continuous semigroup of bounded linear operators $\left\lbrace \mathcal{T}(t)\right\rbrace _{t \geq 0}$, $A_{1} \in \mathcal{L}(X)$ and the initial function $\varphi(\cdot)\in \mathbb{C}\left( [-\tau,0],X\right)$. Then, the classical (strong) solution $u_{0}(\cdot)\in \mathbb{C}\left( [-\tau,\infty),X\right) \cap \mathbb{C}^{1}\left( [0,\infty),X\right)$ of the abstract Cauchy problem \eqref{main equation} for a linear homogeneous functional evolution equation is satisfying $u_{0}(t)\in \mathcal{D}\left(A_{0}\right)$ for all $t\geq 0$ with $\varphi(0)\in \mathcal{D}\left(A_{0}\right)$ and it can be represented in the integral form:
	\begin{equation}\label{sol-2}
		u_{0}(t)=\mathcal{S}(t;\tau)\varphi(0)+
		\int_{-\tau}^{0}\mathcal{S}(t-\tau-s;\tau)A_{1}\varphi(s)\mathrm{d}s, \quad t \geq 0.
	\end{equation}
\end{theorem}
\begin{proof}
	Indeed, by performing integration by parts, we derive that
	\begin{align}\label{s}
		\int_{-\tau}^{0}\mathcal{S}(t-s;\tau)\Big[\varphi^{\prime}(s)-A_{0}\varphi(s)\Big]\mathrm{d}s&=\mathcal{S}(t;\tau)\varphi(0)-\mathcal{S}(t+\tau;\tau)\varphi(-\tau)\nonumber\\
		&-\int_{-\tau}^{0}\Big[\frac{\partial}{\partial s} \mathcal{S}(t-s;\tau)-\mathcal{S}(t-s;\tau)A_{0}\Big]\varphi(s)\mathrm{d}s\nonumber\\
		&=\mathcal{S}(t;\tau)\varphi(0)-\mathcal{S}(t+\tau;\tau)\varphi(-\tau)\nonumber\\
		&+\int_{-\tau}^{0} \mathcal{S}(t-\tau-s;\tau)A_{1}\varphi(s)\mathrm{d}s.
	\end{align}
	By putting \eqref{s} in \eqref{sol-1}, one can reduce the expression \eqref{sol-1} to the form \eqref{sol-2}.
\end{proof}
\begin{remark}
	It is interesting to note that these results are the natural extension of the results are attained in \cite{mahmudov-almatarneh} for first-order impulsive time-delay linear systems.
\end{remark}

Secondly, we consider the following abstract Cauchy problem for a linear non-homogeneous functional evolution equation in a Banach space $X$:
\begin{equation}\label{non-hom}
	\begin{cases}
		\frac{\mathrm{d}}{\mathrm{d}t}u(t)=A_{0}u(t)+A_{1}u(t-\tau)+g(t), \quad t\geq 0,\\
		u(t)=\varphi(t), \quad -\tau\leq t \leq 0,
	\end{cases}
\end{equation}
where $A_{0}:\mathcal{D}\left(A_{0} \right)\subseteq X\to X$ is the infinitesimal generator of a strongly continuous semigroup of bounded linear operators $\left\lbrace \mathcal{T}(t)\right\rbrace _{t \geq 0}$, $A_{1}\in \mathcal{L}\left(X\right)$, $\tau$ is a positive time-delay, $\varphi(\cdot):[-\tau,0]\to X$ is a continuous operator-valued function on $[-\tau,0]$ and $g(\cdot):[0,\infty)\to X$ is a  continuously differentiable operator-valued function on $[0,\infty)$.

We will apply a \textit{superposition principle} for a construction of solution to
the linear non-homogeneous abstract differential equation \eqref{non-hom} with a constant delay.
This principle is clear from linearity of the problem, but it is a very useful principle in solving linear non-homogeneous  differential equations.

\textbf{Superposition principle:} \textit{It is known that if $u_{0}(t)$ is a solution of linear homogeneous abstract Cauchy problem \eqref{main equation} with non-homogeneous initial conditions
	$u_{0}(t)=\varphi(t)$, $-\tau \leq t\leq 0$ and $u_{1}(t)$ is a solution of linear non-homogeneous abstract Cauchy problem \eqref{non-hom} with homogeneous initial conditions $u_{1}(t)=0$, $-\tau \leq t\leq 0$,  then $u(t)\coloneqq u_{0}(t)+u_{1}(t)$ is a solution of linear non-homogeneous abstract Cauchy problem \eqref{non-hom}  with non-homogeneous initial conditions $u(t)=\varphi(t)$, $-\tau \leq t \leq 0$.}

Before finding a closed-form of solution to \eqref{non-hom}, we prove the following auxiliary lemma plays a crucial role in the proof of Theorem \ref{thm-2}.
\begin{lemma}\label{lemma}
	Let $\left\lbrace \mathcal{T}(t)\right\rbrace_{t \geq 0}$ be a strongly continuous semigroup of bounded linear operators. For continuous $f(t)$ on $[0,\infty)$ to $X$, $g(t)=\int_{0}^{t}\mathcal{T}(t-s)f(s)\mathrm{d}s=\int_{0}^{t}\mathcal{T}(s)f(t-s)\mathrm{d}s$ exists and is itself continuous on $[0,\infty)$ to $X$.
	If $f(t)$ is continuously differentiable on $[0,\infty)$, then $g(t)$ is also continuously differentiable on $[0,\infty)$ and the following relations hold true on $[0,\infty)$:
	\begin{align}\label{Leibniz}
		\frac{\mathrm{d}g(t)}{\mathrm{d}t}
		&=\mathcal{T}(t)f(0)+\int_{0}^{t}\mathcal{T}(t-s)f^{\prime}(s)\mathrm{d}s,\\
		&=f(t)+A_{0}\int_{0}^{t}\mathcal{T}(t-s)f(s)\mathrm{d}s. \label{leibniz}
	\end{align}
	\begin{proof}
		Since $\mathcal{T}(t)x$ is continuous with respect to $t$ on $[0,\infty)$ for a fixed $x \in X$, it is obvious that $\mathcal{T}(t-s)f(s)$ is also continuous in $s \in [0,t]$ whenever the same is true of $f(s)$. In this case, $g(t)=\int_{0}^{t}\mathcal{T}(t-s)f(s)\mathrm{d}s$ will exists in the strong topology and be equal to $\int_{0}^{t}\mathcal{T}(s)f(t-s)\mathrm{d}s$. 
		
		Since $f(s)$ is continuously differentiable for any $s \in [0,t]$, by Leibniz integral rule, we obtain that
		\begin{align}\label{A}
			\frac{\mathrm{d}g(t)}{\mathrm{d}t}=\frac{\mathrm{d}}{\mathrm{d}t}\int_{0}^{t}\mathcal{T}(s)f(t-s)\mathrm{d}s&=\mathcal{T}(t)f(0)+\int_{0}^{t}\mathcal{T}(s)\frac{\partial}{\partial
				t}f(t-s)\mathrm{d}s\nonumber\\
			&=\mathcal{T}(t)f(0)+\int_{0}^{t}\mathcal{T}(t-s)f^{\prime}(s)\mathrm{d}s, \quad t\geq 0.
		\end{align}
		Next, we prove the equivalence of the equations \eqref{Leibniz} and \eqref{leibniz}. Since  $\int_{0}^{t}\mathcal{T}(t-s)f(s)\mathrm{d}s\in \mathcal{D}\left( A_{0}\right)$, by the integration by parts formula, it follows that 
		\begin{equation}\label{B}
			\int_{0}^{t}\mathcal{T}(t-s)f^{\prime}(s)\mathrm{d}s
			=f(t)-\mathcal{T}(t)f(0)+A_{0}\int_{0}^{t}\mathcal{T}(t-s)f(s)\mathrm{d}s, \quad t \geq 0.
		\end{equation}
		As a consequence, from \eqref{A} and \eqref{B}, we attain the desired result:
		\begin{equation*}
			\frac{\mathrm{d}g(t)}{\mathrm{d}t}=f(t)+A_{0}\int_{0}^{t}\mathcal{T}(t-s)f(s)\mathrm{d}s, \quad t\geq 0.
		\end{equation*}
		Further, the continuity of $\mathrm{d}g(t)/\mathrm{d}t$ follows from the first part of the lemma using the first representation of $\mathrm{d}g(t)/\mathrm{d}t$ given by \eqref{A}. The proof is complete.
	\end{proof}
\end{lemma}

\begin{theorem}\label{thm-2}
	Let $A_{0}$ be the infinitesimal generator of a strongly continuous semigroup of bounded linear operators $\left\lbrace\mathcal{T}(t)\right\rbrace_{t \geq 0}$. Let $A_{1}\in \mathcal{L}\left(X \right)$ and $g(t)$ be continuously differentiable function on $[0,\infty)$ to $X$. Then there exists a unique continuously differentiable solution $u(\cdot):[-\tau,\infty)\to X$ of the abstract Cauchy problem \eqref{non-hom} for a linear non-homogeneous delay evolution equation which is satisfying $u(t)\in \mathcal{D}\left(A_{0}\right)$ for all $t\geq 0$ with $\varphi(0)\in \mathcal{D}\left(A_{0}\right)$. This solution has a closed form:
	\begin{equation}\label{iki}
		u(t)=u_{0}(t)+u_{1}(t), \quad t\geq 0,
	\end{equation}
	where  $u_{0}(t)$ is given by  \eqref{sol-2}, and 
	\begin{align*}
		&u_{1}(t)=\sum_{n=0}^{\infty}w_{n}(t,n\tau)\mathds{1}_{t \geq n\tau}, \quad t\geq 0,\\
		&w_{0}(t,0)=\int_{0}^{t}\mathcal{T}(t-s)g(s)\mathrm{d}s, \quad t\geq 0,\\
		&w_{n}(t,n\tau)=\int_{n\tau}^{t}\mathcal{T}(t-s)A_{1}w_{n-1}(s-\tau,(n-1)\tau)\mathrm{d}s, \quad t\geq n\tau, \quad n\in \mathbb{N}.
	\end{align*}
\end{theorem}

\begin{proof}
	It follows from Lemma \ref{lemma} that $w_{0}(t,0)$ is continuously differentiable on $[0,\infty)$ and hence by induction that $w_{n}(t,n\tau)\mathds{1}_{t \geq n\tau}$ is like-wise for any $n\in\mathbb{N}$ on $[0,\infty)$. In fact, by Lemma \ref{lemma}, we have
	\begin{align*}
		&w_{0}^{\prime}(t,0)=\mathcal{T}(t)g(0)+\int_{0}^{t}\mathcal{T}(t-s)g^{\prime}(s)\mathrm{d}s, \quad t\geq 0,\\
		&w_{n}^{\prime}(t,n\tau)=\int_{n\tau}^{t}\mathcal{T}(t-s)\mathcal A_{1}w_{n-1}^{\prime}(s-\tau,(n-1)\tau)\mathrm{d}s, \quad t\geq n\tau, \quad n \in \mathbb{N}.
	\end{align*}
	By using \eqref{exponentiall} it is now easy to acquire the following estimations for $t\geq n\tau$, $n\in\mathbb{N}_{0}$:
	\begin{align*}
		&\norm{w_{n}(t,n\tau)}\leq M\Big( M\norm{A_{1}}\exp(-\omega \tau)\Big)^{n}N\exp(\omega t)\frac{(t-n\tau)^{n+1}}{(n+1)!},\\
		&\norm{w_{n}^{\prime}(t,n\tau)}\leq M\Big( M\norm{A_{1}}\exp(-\omega \tau)\Big)^{n}N\exp(\omega t)\Big[\frac{(t-n\tau)^{n}}{n!}+\frac{(t-n\tau)^{n+1}}{(n+1)!}\Big],
	\end{align*}
	where $N\coloneqq \sup\limits_{t \geq 0} \left\lbrace \norm{g(t)},\norm{g^{\prime}(t)}\right\rbrace$.
	
	From these bounds, as in Theorem \ref{main-hom}, this implies that the functional series $\sum_{n=0}^{\infty}w_{n}(t,n\tau)\mathds{1}_{t\geq n\tau}$ and $\sum_{n=0}^{\infty}w_{n}^{\prime}(t,n\tau)\mathds{1}_{t\geq n\tau}$ converge uniformly in each compact subset of $[0,\infty)$ to continuous functions which are $u_{1}(t)$ and $u_{1}^{\prime}(t)$, respectively.  Furthermore, $u_{1}(t)=0$, for $-\tau \leq t \leq 0$. From the definition of $w_{n}(t,n\tau)$, $n\in \mathbb{N}_{0}$ and uniform convergence of the series $\sum_{n=0}^{\infty}w_{n}(t,n\tau)\mathds{1}_{t \geq n\tau}$ in every finite interval, it follows that
	\allowdisplaybreaks
	\begin{align}\label{special}
		u_{1}(t)&=\sum_{n=0}^{\infty}w_{n}(t,n\tau)\mathds{1}_{t \geq n\tau}=w_{0}(t,0)+\sum_{n=1}^{\infty}w_{n}(t,n\tau)\mathds{1}_{t\geq n\tau}\nonumber\\
		&=w_{0}(t,0)+
		\sum_{n=1}^{\infty}\int_{0}^{t}\mathcal{T}(t-s)A_{1}w_{n-1}(s-\tau,(n-1)\tau)\mathds{1}_{s\geq n\tau}\mathrm{d}s\nonumber\\
		&=w_{0}(t,0)+
		\int_{0}^{t}\mathcal{T}(t-s)A_{1}\sum_{n=1}^{\infty}w_{n-1}(s-\tau,(n-1)\tau)\mathds{1}_{s\geq n\tau}\mathrm{d}s\nonumber\\
		&=w_{0}(t,0)+
		\int_{0}^{t}\mathcal{T}(t-s)A_{1}\sum_{n=0}^{\infty}w_{n}(s-\tau,n\tau)\mathds{1}_{s\geq (n+1)\tau}\mathrm{d}s\nonumber\\
		&=w_{0}(t,0)+
		\int_{0}^{t}\mathcal{T}(t-s)A_{1}u_{1}(s-\tau)\mathrm{d}s, \quad t\geq 0.
	\end{align}
	Since $u_{1}(t)$ is continuously differentiable on $[0,\infty)$, by Lemma \ref{lemma}, we can differentiate \eqref{special} term-wise as follows:
	\begin{align*}
		\frac{\mathrm{d}}{\mathrm{d}t}u_{1}(t)
		&=\frac{\mathrm{d}}{\mathrm{d}t}\Big(w_{0}(t,0)+
		\int_{0}^{t}\mathcal{T}(t-s)A_{1}u_{1}(s-\tau)\mathrm{d}s\Big)\\
		&=\frac{\mathrm{d}}{\mathrm{d}t}\Big(\mathcal{T}(t)\ast g(t)\Big)+ \frac{\mathrm{d}}{\mathrm{d}t}\Big(\mathcal{T}(t)\ast \left(A_{1}u_{1}(t-\tau)\right) \Big)\\
		&=A_{0}\mathcal{T}(t)\ast g(t)+g(t)
		+A_{0}\mathcal{T}(t)\ast \left( A_{1}u_{1}(t-\tau)\right)+ A_{1}u_{1}(t-\tau)\\
		&=A_{0}\Big[w_{0}(t,0)+\int_{0}^{t}\mathcal{T}(t-s)A_{1}u_{1}(s-\tau)\mathrm{d}s\Big]+A_{1}u_{1}(t-\tau)+g(t)\\
		&=A_{0}u_{1}(t)+A_{1}u_{1}(t-\tau)+g(t), \quad t \geq 0.
	\end{align*}
	This shows immediately that $u_{1}(t)$ is a \textit{particular solution} of linear non-homogeneous abstract Cauchy problem \eqref{non-hom}. In other words, $u_{1}(t)$ is a solution of \eqref{non-hom} with \textit{zero} initial conditions, i.e., $u_{1}(t)=0$, $-\tau \leq t \leq 0$. Therefore, by Theorem \ref{main-hom}, $u(t)=u_{0}(t)+u_{1}(t)$ is a solution for linear non-homogeneous abstract Cauchy problem \eqref{non-hom}. The uniqueness of a \textit{particular solution} follows precisely as in the uniqueness proof of Theorem \ref{main-hom}. The proof is complete.
\end{proof}

\begin{remark}
	Note that the particular solution of \eqref{non-hom} has can also expressed by
	\begin{equation}\label{son-2}
		u_{1}(t)=\sum_{n=0}^{\infty}w_{n}(t,n\tau), \quad t\geq 0,
	\end{equation}
	where 
	\begin{align*}
		w_{0}(t,0)&=\int_{0}^{t}\mathcal{T}(t-s)g(s)\mathrm{d}s, \quad t\geq0,\\
		w_{n}(t,n\tau)
		&=\int_{0}^{t-n\tau}w_{n-1}(t-s-\tau,(n-1)\tau)A_{1}\mathcal{T}(s)\mathrm{d}s, \quad t\geq n\tau, \quad n \in \mathbb{N}.
	\end{align*}
\end{remark}

If we consider linear non-homogeneous abstract differential equation with a discrete delay \eqref{non-hom} on $[-\tau,T]$, then we can introduce a \textit{piece-wise} construction for a \textit{particular solution} of \eqref{non-hom} as follows.

\begin{corollary}
	Let $\mathcal{T}(t)$ be a strongly continuous semigroup of bounded linear operators with infinitesimal generator $A_{0}$  on $[0,T]$. Let $A_{1}\in \mathcal{L}\left(X \right)$ and $g(t)$ be continuously differentiable operator-valued function on $[0,T]$ to $X$.
	Then, there exists a unique continuously differentiable function $u_{1}(\cdot):[0,T]\to X$ is a particular solution of \eqref{non-hom} with zero initial conditions $u_{1}(t)=0$, $-\tau \leq t \leq 0$ and satisfying $u(t)\in \mathcal{D}\left(A_{0}\right)$ for all $0\leq t\leq T$. This solution has a closed form:
	\begin{equation}\label{son-2}
		u_{1}(t)=\sum_{k=0}^{n}w_{k}(t,k\tau), \quad n\tau < t\leq (n+1)\tau, \quad n\in \mathbb{N}_{0},
	\end{equation}
	where 
	\begin{align*}
		w_{0}(t,0)&=\int_{0}^{t}\mathcal{T}(t-s)g(s)\mathrm{d}s, \quad t\geq0,\\
		w_{k}(t,k\tau)&=\int_{k\tau}^{t}\mathcal{T}(t-s)A_{1}w_{k-1}(s-\tau,(k-1)\tau)\mathrm{d}s\\
		&=\int_{0}^{t-k\tau}w_{k-1}(t-s-\tau,(k-1)\tau)A_{1}\mathcal{T}(s)\mathrm{d}s, \quad t\geq k\tau, \quad k=1,2,\ldots,n.
	\end{align*}
\end{corollary}

The particular solution $u_{1}(t)$, $t \geq 0$ of \eqref{non-hom} can also be put in a more suggestive form. By making use of the series representation of \textit{fundamental solution} to the abstract Cauchy problem  \eqref{main equation} and Fubini's theorem, we attain:
\begin{align*}
	&\int_{0}^{t}S_{n}(t-s,n\tau)\mathds{1}_{t-s \geq n\tau}g(s)\mathrm{d}s\\
	=&\int_{0}^{t}\int_{0}^{t-s}\mathcal{T}(t-s-\sigma)A_{1}S_{n-1}(\sigma-\tau,(n-1)\tau)\mathds{1}_{\sigma \geq n\tau}g(s) \mathrm{d}\sigma\mathrm{d}s\\
	=&\int_{0}^{t}\int_{s}^{t}\mathcal{T}(t-\sigma)A_{1}S_{n-1}(\sigma-s-\tau,(n-1)\tau)\mathds{1}_{\sigma-s \geq n\tau}g(s) \mathrm{d}\sigma\mathrm{d}s\\
	=&\int_{0}^{t}\mathcal{T}(t-\sigma)A_{1}
	\Big[\int_{0}^{\sigma}S_{n-1}(\sigma-s,n\tau)\mathds{1}_{\sigma-s \geq n\tau}g(s)\mathrm{d}s\Big]\mathrm{d}\sigma, \quad n\in \mathbb{N}, \quad t\geq 0.
\end{align*}
Hence, by mathematical induction principle, we get  
\begin{equation*}
	w_{n}(t,n\tau)\mathds{1}_{t\geq n\tau}=\int_{0}^{t}S_{n}(t-s,n\tau)\mathds{1}_{t-s \geq n\tau}g(s)\mathrm{d}s, \quad n\in\mathbb{N}, \quad t\geq 0.
\end{equation*}
Finally, because of the \textit{uniform convergence}  of the series in the \textit{uniform operator topology}, we derive that
\begin{align}\label{last}
	u_{1}(t)&=\sum_{n=0}^{\infty}w_{n}(t,n\tau)\mathds{1}_{t\geq n\tau}\nonumber\\
	&=\int_{0}^{t}\mathcal{T}(t-s)g(s)\mathrm{d}s+
	\sum_{n=1}^{\infty}\int_{0}^{t}S_{n}(t- s,n\tau)\mathds{1}_{t-s \geq n\tau}g(s)\mathrm{d}s\nonumber\\
	&=\sum_{n=0}^{\infty}\int_{0}^{t}S_{n}(t- s,n\tau)\mathds{1}_{t-s \geq n\tau}g(s)\mathrm{d}s\nonumber\\
	&=\int_{0}^{t}\sum_{n=0}^{\infty}S_{n}(t- s,n\tau)\mathds{1}_{t-s\geq n\tau}g(s)\mathrm{d}s\nonumber\\
	&=\int_{0}^{t}\mathcal{S}(t-s;\tau)g(s)\mathrm{d}s, \quad t\geq 0.
\end{align}

\section{Delay evolution equations with bounded linear operators}\label{sect:bounded}
In this section, we consider different important cases of delay evolution equation \eqref{main equation} with bounded linear operators $A_{0},A_{1}\in\mathcal{L}\left( X\right)$ in a Banach space $X$. We will take a \textit{domain of equation} as $[0,T]$ where $T=(n+1)\tau$ for a fixed $n\in\mathbb{N}_{0}$ and use a \textit{piece-wise} construction for the \textit{delayed} operator-valued functions. It is known that, in this case, the domain $\mathcal{D}\left(A_{0}\right)$ coincides with the state space $X$, i.e., $\mathcal{D}\left(A_{0}\right)=X$, and a one-parameter $C_{0}$-semigroup family of bounded linear operators $\mathcal{T}(t)$, $0\leq t\leq T$  which is continuous with respect to \textit{uniform} operator topology defined by
\begin{equation}\label{exp}
	\mathcal{T}(t)=\exp\left( A_{0}t\right) =\sum_{k=0}^{\infty}A^{k}_{0}\frac{t^{k}}{k!}.
\end{equation}

For a fixed $n\in \mathbb{N}$ and time-delay $\tau>0$, we define the following sequence of operator-valued functions via a \textit{recursive} way:
\begin{align}\label{seq}
	&S_{0}(t,0)\coloneqq \exp\left( A_{0}t\right) , \quad t \geq 0,\nonumber\\
	&S_{1}(t,\tau)\coloneqq\begin{cases}
		\int_{\tau}^{t}\exp\left( A_{0}(t-s)\right)A_{1}	S_{0}(s-\tau,0)\mathrm{d}s, \quad t \geq \tau, \\
		\Theta, \quad t < \tau,
	\end{cases}\nonumber\\
	& \ldots \qquad \ldots \qquad \ldots \qquad \ldots \qquad \ldots \qquad \ldots \qquad \ldots \qquad \ldots \qquad \ldots\nonumber\\
	&S_{n}(t,n\tau)\coloneqq\begin{cases}
		\int_{n\tau}^{t}\exp\left( A_{0}(t-s)\right) A_{1}	S_{n-1}(s-\tau,(n-1)\tau)\mathrm{d}s, \quad t\geq n\tau, \\
		\Theta, \quad t < n\tau.
	\end{cases}
\end{align}

Therefore, in this case, the \textit{fundamental solution} of abstract Cauchy problem \eqref{main equation} with bounded linear operators which is represented by \textit{delayed Dyson-Phillips series} becomes the \textit{delayed perturbation of an operator-valued exponential function}.

\begin{definition}
	The delayed perturbation of an operator-valued exponential function $\mathcal{S}(\cdot;\tau):\mathbb{R}\to\mathcal{L}\left(X \right)$ generated by  linear operators $A_{0},A_{1} \in\mathcal{L}\left(X\right) $ is defined by
	\begin{equation}\label{defn}
		\mathcal{S}(t;\tau)\coloneqq\begin{cases}
			\Theta, \quad  -\infty<t<0, \\
			I, \quad t=0, \\
			\exp\left( A_{0}t\right) +S_{1}(t,\tau)+\dots+
			S_{n}(t,n\tau), \quad n\tau< t \leq  (n+1)\tau, \quad  n\in \mathbb{N}_{0}.
		\end{cases}
	\end{equation}
\end{definition}

Depending on the relation between bounded linear operators $A_{0}$ and $A_{1}$, one can obtain various representation formulae for the \textit{delayed perturbation of an operator-valued exponential function}.

The following theorem deals with the fundamental solution of delay evolution equation \eqref{main equation} with \textit{non-permutable} bounded linear operators.
\allowdisplaybreaks
\begin{theorem}\label{non-permutable-1}
	Let $A_{0},A_{1} \in \mathcal{L}(X)$ with non-zero commutator $[A_{0},A_{1}]\coloneqq A_{0}A_{1}-A_{1}A_{0}\neq 0$. Then a fundamental solution  $\mathcal{S}(\cdot;\tau):[-\tau,T] \to \mathcal{L}\left(X \right)$ of linear homogeneous evolution equation with a constant delay \eqref{main equation} which is satisfying the following initial conditions
	\begin{equation}\label{data}
		\mathcal{S}(t;\tau)=\begin{cases}
			\Theta,\quad  -\tau \leq t <0,\\
			I, \quad t=0,
		\end{cases}	
	\end{equation}
	can be represented by
	
	\begin{equation}\label{formul}
		\mathcal{S}(t;\tau)=\sum_{k=l}^{\infty}\sum_{l=0}^{n}\mathcal{Q}_{k+1}(l\tau)\frac{(t-l\tau)^{k}}{k!}, \quad n\tau < t \leq  (n+1)\tau, \quad n \in \mathbb{N}_{0},
	\end{equation}
	where a linear operator family $\left\lbrace \mathcal{Q}_{k+1}(l\tau): k,l\in \mathbb{N}_{0}\right\rbrace \subset  \mathcal{L}\left( X\right) $ is given by
	\begin{align}\label{recursive-2}
		\mathcal{Q}_{k+1}(0)\coloneqq A_{0}^{k}, \quad k\in \mathbb{N}_{0}, \quad  \mathcal{Q}_{k+1}(l\tau)\coloneqq \sum_{m=l}^{k}A_{0}^{k-m}A_{1}
		\mathcal{Q}_{m}((l-1)\tau), \quad k,l\in \mathbb{N}.
	\end{align}
\end{theorem}

\begin{proof}
	By making use of the recursive sequence of operator-valued functions \eqref{seq}, we derive the basis case for $n=0$:
	\begin{align*}
		S_{0}(t,0)=\exp\left( A_{0}t\right) =\sum_{k=0}^{\infty}A_{0}^{k}\frac{t^{k}}{k!}=&\sum_{k=0}^{\infty}\mathcal{Q}_{k+1}(0)\frac{t^{k}}{k!}, \quad t \geq 0, \\
		&\mathcal{Q}_{k+1}(0)\coloneqq A_{0}^{k}, \quad k=0,1,\ldots
	\end{align*}
	\allowdisplaybreaks
	For $n=1$, by making use of the well-known Cauchy product formula for double infinite series and interchanging the order of summation and integration which is permissible in accordance with the uniform convergence of the series \eqref{exp}, we get: 
	\begin{align*}
		S_{1}(t,\tau)&=\int_{\tau}^{t}\exp\left( A_{0}(t-s)\right) A_{1}\exp\left( A_{0}(s-\tau)\right) \mathrm{d}s\\
		&=\sum_{k=0}^{\infty}\sum_{m=0}^{\infty}A_{0}^{k}A_{1}A_{0}^{m}\int_{\tau}^{t}\frac{(t-s)^{k}}{k!}\frac{(s-\tau)^{m}}{m!}\mathrm{d}s\\
		&=\sum_{k=0}^{\infty}\sum_{m=0}^{\infty}A_{0}^{k}A_{1}A_{0}^{m}\frac{(t-\tau)^{k+m+1}}{(k+m+1)!}\\
		&=\sum_{k=0}^{\infty}\sum_{m=0}^{k}A_{0}^{k-m}A_{1}A_{0}^{m}\frac{(t-\tau)^{k+1}}{(k+1)!}\\
		&=\sum_{k=1}^{\infty}\sum_{m=0}^{k-1}A_{0}^{k-m-1}A_{1}A_{0}^{m}\frac{(t-\tau)^{k}}{k!}\\
		&=\sum_{k=1}^{\infty}\sum_{m=1}^{k}A_{0}^{k-m}A_{1}A_{0}^{m-1}\frac{(t-\tau)^{k}}{k!}\\
		&=\sum_{k=1}^{\infty}\sum_{m=1}^{k}A_{0}^{k-m}A_{1}\mathcal{Q}_{m}(0)\frac{(t-\tau)^{k}}{k!}\\
		&=\sum_{k=1}^{\infty}\mathcal{Q}_{k+1}(\tau)\frac{(t-\tau)^{k}}{k!}, \quad t \geq \tau,\\
		& \hspace{0.2cm} \mathcal{Q}_{k+1}(\tau)\coloneqq \sum_{m=1}^{k}A_{0}^{k-m}A_{1}\mathcal{Q}_{m}(0),\quad k=1,2,\ldots
	\end{align*}
	In a recursive way, for $n=2$, by using $\mathcal{Q}_{1}(\tau)=\Theta$, we assure that
	\begin{align*}
		S_{2}(t,2\tau)&=\int_{2\tau}^{t}\exp\left( A_{0}(t-s)\right) A_{1}S_{1}(s-\tau,\tau)\mathrm{d}s\\
		&=\sum_{k=0}^{\infty}\sum_{m=1}^{\infty}A_{0}^{k}A_{1}\mathcal{Q}_{m+1}(\tau)\int_{2\tau}^{t}\frac{(t-s)^{k}}{k!}\frac{(s-2\tau)^{m}}{m!}\mathrm{d}s\\
		&=\sum_{k=0}^{\infty}\sum_{m=1}^{\infty}A_{0}^{k}A_{1}\mathcal{Q}_{m+1}(\tau)\frac{(t-2\tau)^{k+m+1}}{(k+m+1)!}\\
		&=\sum_{k=1}^{\infty}\sum_{m=1}^{k}A_{0}^{k-m}A_{1}\mathcal{Q}_{m+1}(\tau)\frac{(t-2\tau)^{k+1}}{(k+1)!}\\
		&=\sum_{k=2}^{\infty}\sum_{m=1}^{k-1}A_{0}^{k-m-1}A_{1}\mathcal{Q}_{m+1}(\tau)\frac{(t-2\tau)^{k}}{k!}\\
		&=\sum_{k=2}^{\infty}\sum_{m=2}^{k}A_{0}^{k-m}A_{1}\mathcal{Q}_{m}(\tau)\frac{(t-2\tau)^{k}}{k!}\\
		&=\sum_{k=2}^{\infty}\mathcal{Q}_{k+1}(2\tau)\frac{(t-2\tau)^{k}}{k!}, \quad t \geq 2\tau,\\
		& \hspace{0.2 cm} \mathcal{Q}_{k+1}(2\tau)\coloneqq \sum_{m=2}^{k}A_{0}^{k-m}A_{1}\mathcal{Q}_{m}(\tau),\quad k=2,3,\ldots
	\end{align*}
	Recursively, for the $n$-th case, it yields that 
	\begin{align*}
		S_{n}(t,n\tau)&=\int_{n\tau}^{t}\exp\left( A_{0}(t-s)\right) A_{1}S_{n-1}(s-\tau,(n-1)\tau)\mathrm{d}s\\
		&=\sum_{k=n}^{\infty}\mathcal{Q}_{k+1}(n\tau)\frac{(t-n\tau)^{k}}{k!},  \quad t\geq n\tau,\\ &\hspace{0.2 cm}\mathcal{Q}_{k+1}(n\tau)\coloneqq \sum_{m=n}^{k}A_{0}^{k-m}A_{1}\mathcal{Q}_{m}((n-1)\tau),\quad k=n,n+1, \ldots
	\end{align*}
	
	Therefore, the fundamental solution $\mathcal{S}(\cdot;\tau): [-\tau,T] \to \mathcal{L}\left(X \right)$  of \eqref{main equation} which is satisfying the initial conditions \eqref{data} is a summation of the finite number of above terms and we attain the following desired result:
	\begin{align*}
		\mathcal{S}(t;\tau)&=\sum_{k=0}^{\infty}\mathcal{Q}_{k+1}(0)\frac{t^{k}}{k!}+\sum_{k=1}^{\infty}\mathcal{Q}_{k+1}(\tau)\frac{(t-\tau)^{k}}{k!}+\ldots+\sum_{k=n}^{\infty}\mathcal{Q}_{k+1}(n\tau)\frac{(t-n\tau)^{k}}{k!}\\&=
		\sum_{k=l}^{\infty}\sum_{l=0}^{n}\mathcal{Q}_{k+1}(l\tau)\frac{(t-l\tau)^{k}}{k!}, \quad n\tau< t \leq  (n+1)\tau, \quad n \in \mathbb{N}_{0}.
	\end{align*}
	\allowdisplaybreaks
	The proof is complete.
\end{proof}

\begin{remark}
	It should be noted that this particular case is a natural extension of the results obtained by Mahmudov in \cite{mahmudov} concerning non-commutative matrix coefficients for time-delay systems of fractional order.
\end{remark}

A linear operator family $\left\lbrace \mathcal{Q}_{k+1}(l\tau): k,l\in \mathbb{N}_{0}\right\rbrace \subset  \mathcal{L}\left( X\right)$ plays a role as a kernel for the delayed perturbation of operator-valued exponential functions. With the help of formulae \eqref{recursive-2}, simple calculations show that

\begin{center}
	\begin{tabular}{| l | l | l | l | l | l | p{0.7cm} |}
		\hline
		$\mathcal{Q}_{k+1}(l\tau)$ & $l=0$ & $l=1$ & $l=2$ & $l=3$ & \ldots&  $l=n$  \\ \hline
		$k=0$ & $I$ & $\Theta$ & $\Theta$ & $\Theta$ & \ldots & $\Theta$ \\ \hline
		$k=1$ & $A_{0}$ & $A_{1}$ & $\Theta$ & $\Theta$ & \ldots & $\Theta$ \\ \hline
		$k=2$ & $A_{0}^{2}$ & $A_{0}A_{1}+A_{1}A_{0}$ & $A_{1}^{2}$ & $\Theta$ & \ldots & $\Theta$ \\
		\hline
		$k=3$ & $A_{0}^{3}$ & $A_{0}\left( A_{0}A_{1}+A_{1}A_{0}\right) +A_{1}A_{0}^{2}$ & $A_{1}A_{0}^{2}+ A_{1}\left( A_{0}A_{1}+A_{1}A_{0}\right)$& $A_{1}^{3}$ & \ldots & $\Theta$ \\
		\hline	
		\ldots & \ldots & \ldots & \ldots & \ldots & \ldots & \ldots \\ \hline
		$k=n$ & $A_{0}^{n}$ & \ldots &\ldots & \ldots & \ldots & $A_{1}^{n}$ \\ \hline
		\ldots & \ldots & \ldots & \ldots & \ldots & \ldots & \ldots \\ \hline
	\end{tabular}
\end{center}

From this table, we can derive the following important result:
\begin{itemize}
	\item If $l\geq k+1$ for $k\geq 0$, then $\mathcal{Q}_{k+1}(l\tau)=\Theta$. Therefore, a linear operator family $\left\lbrace \mathcal{Q}_{k+1}(l\tau): k,l\in \mathbb{N}_{0}\right\rbrace \subset  \mathcal{L}\left( X\right)$ is a \textit{lower triangular operator-matrix};
	\item If $A_{1}=\Theta$, then for $k\in \mathbb{N}_{0}$, we have: 
	\begin{equation*}
		\mathcal{Q}_{k+1}(l\tau)=\begin{cases}
			A_{0}^{k},\quad l=0,\\
			\Theta , \hspace{0.55 cm} l \in \mathbb{N},
		\end{cases}
	\end{equation*}
	and a \textit{fundamental solution} becomes an operator-valued exponential function: 
	\begin{equation*}
		\mathcal{S}(t;\tau)=\exp\left( A_{0}t\right), \quad t\geq 0.
	\end{equation*}
	\item If $A_{0}=\Theta$, then for $k,l \in \mathbb{N}_{0}$, we have:
	\begin{equation*}
		\mathcal{Q}_{k+1}(l\tau)=\begin{cases}
			A_{1}^{l},\quad k=l,\\
			\Theta , \hspace{0.55 cm} k \neq l,
		\end{cases}
	\end{equation*}
	and a \textit{fundamental solution} becomes the \textit{pure delayed} operator-valued exponential function as below:
	\begin{align*}
		\mathcal{S}(t;\tau)=\exp_{\tau}\left(A_{1}t\right)&=\mathcal{Q}_{1}(0)+\mathcal{Q}_{2}(\tau)(t-\tau)+\ldots+\mathcal{Q}_{n+1}(n\tau)\frac{(t-n\tau)^{n}}{n!}\\
		&=I+A_{1}(t-\tau)+\ldots+A_{1}^{n}\frac{(t-n\tau)^{n}}{n!}\\
		&=\sum_{l=0}^{n}A_{1}^{l}\frac{(t-l\tau)^{l}}{l!}, \quad n\tau< t \leq  (n+1)\tau, \quad n\in \mathbb{N}_{0}.
	\end{align*}
\end{itemize}

Note that a \textit{piece-wise} construction of the \textit{pure delayed} operator-valued exponential function $\exp_{\tau}\left(A_{1}\cdot\right):[-\tau,T]\to \mathcal{L}\left(X\right)$ generated by a linear operator $A_{1}\in \mathcal{L}\left(X\right)$ can be written via the following \textit{explicit} formula: 
\begin{equation}\label{pure-delayed}
	\exp_{\tau}\left(A_{1}t\right) =\begin{cases}
		\Theta, \quad -\tau \leq t< 0,\\
		I, \quad t=0,\\
		I+A_{1}\left(t-\tau\right)+A_{1}^{2}\frac{(t-2\tau)^{2}}{2!} +\ldots+A_{1}^{n}\frac{(t-n\tau)^{n}}{n!}, \quad n\tau <t \leq (n+1)\tau, \quad n\in \mathbb{N}_{0}.
	\end{cases}
\end{equation}

A linear operator family $\left\lbrace \mathcal{Q}_{k+1}(l\tau): k,l\in \mathbb{N}_{0}\right\rbrace \subset  \mathcal{L}\left( X\right)$ has the following properties depending on the commutativity condition of $A_{0},A_{1}\in \mathcal{L}\left(X \right)$.

\begin{theorem}
	A linear operator family $\left\lbrace \mathcal{Q}_{k+1}(l\tau): k,l\in \mathbb{N}_{0}\right\rbrace \subset  \mathcal{L}\left( X\right)$ has the following properties:
	
	$(i)$ For any (permutable and non-permutable) linear operators $A_{0},A_{1} \in \mathcal{L}\left( X\right)$, we have
	\begin{align}\label{one}
		\mathcal{Q}_{k+1}(l\tau)&=A_{0}\mathcal{Q}_{k}(l\tau)+A_{1}\mathcal{Q}_{k}((l-1)\tau),\\
		\mathcal{Q}_{0}(l\tau)&=\mathcal{Q}_{k}(-\tau)=\Theta, \quad   k,l\in\mathbb{N}_{0} \nonumber.
	\end{align}
	
	$(ii)$ If $A_{0}A_{1}=A_{1}A_{0}$, then we have
	\begin{equation}\label{two}
		\mathcal{Q}_{k+1}(l\tau)=\binom{k}{l}A_{0}^{k-l}A_{1}^{l}, \quad  k,l\in \mathbb{N}_{0}.
	\end{equation}
\end{theorem}
\begin{proof}
	$(i)$ By making use of the mathematical induction principle, we can prove \eqref{one} is true for all $l\in \mathbb{N}_{0}$ for a fixed $k\in \mathbb{N}_{0}$. It is obvious that the formula \eqref{one} holds true for $l=0$. Since $	\mathcal{Q}_{k+1}(0)=A_{0}^{k}$ and $\mathcal{Q}_{k}(-\tau)=\Theta$ for any $k\in \mathbb{N}_{0}$, for $l=0$, we obtain:
	\begin{equation*}
		\mathcal{Q}_{k+1}(0)=A_{0}\mathcal{Q}_{k}(0)+A_{1}\mathcal{Q}_{k}(-\tau).
	\end{equation*}
	
	Suppose that the formula \eqref{one} is true for $l\in \mathbb{N}_{0}$. Then by applying the formula \eqref{recursive-2} for $l$-th case, we prove the statement is true for $(l+1)\in \mathbb{N}_{0}$ as follows:
	\begin{align*}
		\mathcal{Q}_{k+1}((l+1)\tau)&=\sum_{m=l+1}^{k}A_{0}^{k-m}A_{1}\mathcal{Q}_{m}(l\tau)\\
		&=A_{0}\sum_{m=l+1}^{k-1}A_{0}^{k-1-m}A_{1}\mathcal{Q}_{m}(l\tau)+A_{1}\mathcal{Q}_{k}(l\tau)\\
		&=A_{0}\mathcal{Q}_{k}((l+1)\tau)+A_{1}\mathcal{Q}_{k}(l\tau).
	\end{align*}
	This implies that the formula \eqref{one} is true for arbitrary $l\in \mathbb{N}_{0}$.
	
	To show $(ii)$, we will use proof by induction with regard to $l\in \mathbb{N}_{0}
	$ for a fixed $k\in \mathbb{N}_{0}$ via the formula \eqref{recursive-2}. Since $A_{0}A_{1}=A_{1}A_{0}$, for $l=0$, we have
	\begin{equation*}
		\mathcal{Q}_{k}(0)=A_{0}^{k}=\binom{k}{0}A_{0}^{k}A_{1}^{0}. 
	\end{equation*}
	Assume that the formula \eqref{two} is true for $l=n \in \mathbb{N}_{0}$:
	\begin{equation*}
		\mathcal{Q}_{k+1}(n\tau)=\binom{k}{n}A_{0}^{k-n}A_{1}^{n}, \quad  k\in \mathbb{N}_{0}.
	\end{equation*}
	Let us prove it for $l=n+1$ as below:
	\begin{align*}
		\mathcal{Q}_{k+1}((n+1)\tau)&=\sum_{m=n+1}^{k}A_{0}^{k-m}A_{1}\mathcal{Q}_{m}(n\tau)\\
		&=\sum_{m=n+1}^{k}A_{0}^{k-m}A_{1}\binom{m-1}{n}A_{0}^{m-1-n}A_{1}^{n}\\
		&=\sum_{m=n+1}^{k}\binom{m-1}{n}A_{0}^{k-1-n}A_{1}^{n+1}\\
		&=\binom{k}{n+1}A_{0}^{k-(n+1)}A_{1}^{n+1},
	\end{align*}
	where we have used for $k\geq n+1$ the following identity:
	\begin{equation*}
		\sum_{m=n+1}^{k}\binom{m-1}{n}=\binom{n}{n}+\binom{n+1}{n}+\ldots+\binom{k-1}{n}=\binom{k}{n+1}.
	\end{equation*}
	It follows that the formula \eqref{two} is true for any $l\in \mathbb{N}_{0}$. The proof is complete.
\end{proof}

\begin{remark}
	Using the formulae \eqref{one} and \eqref{two} together, we obtain the following relation for permutable bounded linear operators $A_{0}$ and $A_{1}$:
	\begin{align}\label{binom}
		\binom{k}{l}A_{0}^{k-l}A_{1}^{l}&=A_{0}\binom{k-1}{l}A_{0}^{k-l-1}A_{1}^{l}+A_{1}\binom{k-1}{l-1}A_{0}^{k-l-1}A_{1}^{l-1}\nonumber\\
		&=\binom{k-1}{l}A_{0}^{k-l}A_{1}^{l}+\binom{k-1}{l-1}A_{0}^{k-l-1}A_{1}^{l}, \quad   k,l\in\mathbb{N}_{0}.
	\end{align}
	It is easy to see that \eqref{binom} is a Pascal's rule of binomial coefficients for bounded linear operators (especially matrices). Analogously, the \textit{decomposition formula} \eqref{one} is a \textit{generalisation of Pascal's rule} for non-permutable bounded linear operators (especially, matrices).
\end{remark}
\begin{remark}
	It should be noted that a similar operator family $\left\lbrace \mathcal{Q}_{k,l}: k,l\in \mathbb{N}_{0}\right\rbrace \subset  \mathcal{L}\left( X\right)$ is proposed to investigate the properties of solutions to fractional-order multi-term evolution equations and functional evolution equations in \cite{functional-1,functional-2}, respectively. Furthermore, a linear operator family $\left\lbrace \mathcal{Q}_{k,l}: k,l\in \mathbb{N}_{0}\right\rbrace \subset  \mathcal{L}\left( X\right)$  is satisfying the following corresponding properties:
	
	$(i)$ For any (permutable and non-permutable) linear operators $A_{0},A_{1} \in \mathcal{L}\left( X\right)$, we have
	\begin{align*}
		\mathcal{Q}_{k,l}&=A_{0}\mathcal{Q}_{k-1,l}+A_{1}\mathcal{Q}_{k,l-1}, \quad   k,l\in\mathbb{N}_{0}.
	\end{align*}
	
	$(ii)$ If $A_{0}A_{1}=A_{1}A_{0}$, then we have
	\begin{equation*}
		\mathcal{Q}_{k,l}=\binom{k+l}{l}A_{0}^{k}A_{1}^{l}, \quad  k,l\in \mathbb{N}_{0}.
	\end{equation*}
	
	$(iii)$ For non-permutable linear operators $A_{0},A_{1} \in \mathcal{L}\left( X\right)$, we have:
	\begin{equation*}
		\sum_{l=0}^{k}\mathcal{Q}_{k-l,l}=\Big(A_{0}+A_{1}\Big)^{k}, \quad k\in \mathbb{N}_{0}.
	\end{equation*}
	
	$(iv)$ For permutable linear operators $A_{0},A_{1} \in \mathcal{L}\left( X\right)$, we have:
	\begin{equation*}
		\sum_{l=0}^{k}\binom{k}{l}A_{0}^{k-l}A_{1}^{l}=\Big(A_{0}+A_{1}\Big)^{k}, \quad  k\in \mathbb{N}_{0}.
	\end{equation*}
\end{remark}

Therefore, by Theorem \ref{non-permutable-1}, we can use another suggestive explicit formula for the delayed perturbation of an operator-valued exponential function as follows.

\begin{definition}
	The delayed perturbation of an operator-valued exponential function $\mathcal{S}(\cdot;\tau):\mathbb{R} \to \mathcal{L}\left(X \right)$ generated by  linear operators $A_{0},A_{1} \in\mathcal{L}\left(X \right) $ is defined by 
	\begin{equation}\label{defn0}
		\mathcal{S}(t;\tau)\coloneqq\begin{cases}
			\Theta, \quad  -\infty\leq t<0, \\
			I, \quad t=0, \\
			\sum\limits_{k=0}^{\infty}\mathcal{Q}_{k+1}(0)\frac{t^{k}}{k!}+\sum\limits_{k=1}^{\infty}\mathcal{Q}_{k+1}(\tau)\frac{(t-\tau)^{k}}{k!}\\\hspace{2.5cm}+\ldots+
			\sum\limits_{k=n}^{\infty}\mathcal{Q}_{k+1}(n\tau)\frac{(t-n\tau)^{k}}{k!}, \quad n\tau< t\leq  (n+1)\tau, \quad n\in \mathbb{N}_{0}.
		\end{cases}
	\end{equation}
\end{definition}

It is obvious that operator coefficients of the first series of \eqref{defn0} are the elements of the first column ($l=0$) of the above table, operator coefficients of the second series of \eqref{defn0} are elements of the second column $(l=1)$ of the table and so on.

Rearranging the terms we can write the delayed perturbation of an operator-valued exponential function as follows:

\begin{align}\label{shifting}
	\mathcal{S}(t;\tau)&=\mathcal{Q}_{1}(0)+\Big( \mathcal{Q}_{2}(0)t+\mathcal{Q}_{2}(\tau)(t-\tau)\Big) \nonumber \\
	&+\left( \mathcal{Q}_{3}(0)\frac{t^{2}}{2!}+\mathcal{Q}_{3}(\tau)\frac{(t-\tau)^{2}}{2!}+\mathcal{Q}_{3}(2\tau)\frac{(t-2\tau)^{2}}{2!}\right) \nonumber\\
	&+\ldots+\left( \mathcal{Q}_{n+1}(0)\frac{t^{n}}{n!}+
	\mathcal{Q}_{n+1}(\tau)\frac{(t-\tau)^{n}}{n!}+\ldots+
	\mathcal{Q}_{n+1}(n\tau)\frac{(t-n\tau)^{n}}{n!}
	\right) \nonumber\\
	&+\Big( \mathcal{Q}_{n+2}(0)\frac{t^{n+1}}{(n+1)!}+
	\mathcal{Q}_{n+2}(\tau)\frac{(t-\tau)^{n+1}}{(n+1)!}+\ldots+\mathcal{Q}_{n+2}(n\tau)\frac{(t-n\tau)^{n+1}}{(n+1)!}\Big) +\ldots + \nonumber\\
	&=\sum_{k=0}^{\infty}\sum_{l=0}^{k}\mathcal{Q}_{k+1}(l\tau)\frac{(t-l\tau)_{+}^{k}}{k!}, \quad n\tau <t \leq (n+1)\tau,\quad n\in \mathbb{N}_{0}.
\end{align}

We can propose an elegant representation formula for the delayed perturbation of an operator-valued exponential function that is equivalent to the formula \eqref{shifting} and this behaviour helps us to understand the delayed construction of operator-valued functions.

To do this, we introduce a \textit{shift operator} also known as \textit{translation operator} from \textit{operational calculus} is an operator $\mathbb{T}_{\tau}$ for $\tau\in \mathbb{R}$ that takes a function $t\mapsto f(t)$ on $\mathbb{R}$ to its translation $t\mapsto f(t+\tau)$:
\begin{equation*}
	\mathbb{T}_{\tau}f(t) \coloneqq f(t+\tau).
\end{equation*}
For the notation of our results, we will use the following \textit{Lagrange translation formula} for \textit{shift operators}:
\begin{equation*}
	\mathbb{T}_{\tau}f(t)=f(t+\tau)\coloneqq
	\exp\left( \tau \frac{\mathrm{d}}{\mathrm{d}t}\right)  f(t).
\end{equation*}
Using the shift operator we rewrite the explicit formula \eqref{shifting} as below:
\begin{align}\label{shift}
	\mathcal{S}(t;\tau)&=\sum_{k=0}^{\infty}\sum_{l=0}^{k}\mathcal{Q}_{k+1}(l\tau)\frac{(t-l\tau)_{+}^{k}}{k!}\nonumber\\
	&=\sum_{k=0}^{\infty}\sum_{l=0}^{k}\mathcal{Q}_{k+1}(l\tau)\exp\left(-l\tau \frac{\mathrm{d}}{\mathrm{d}t}\right) \frac{(t)_{+}^{k}}{k!},  \quad n\tau <t \leq (n+1)\tau, \quad n\in \mathbb{N}_{0}.
\end{align}

Based on the following lemma, we introduce a new representation formula for the delayed perturbation of an operator-valued exponential function via a \textit{translation operator}. This identity is the delayed analogue of the \textit{generalisation of binomial theorem} for non-permutable linear operators $A_{0},A_{1} \in \mathcal{L}\left( X\right)$.
\allowdisplaybreaks
\begin{lemma}
	A linear operator family $\left\lbrace \mathcal{Q}_{k+1}(l\tau): k,l\in \mathbb{N}_{0}\right\rbrace \subset  \mathcal{L}\left( X\right)$ satisfies the following identity:
	\begin{equation}\label{relat}
		\sum_{l=0}^{k}\mathcal{Q}_{k+1}(l\tau)\exp\left(-l\tau \frac{\mathrm{d}}{\mathrm{d}t}\right) =\Big(A_{0}+A_{1}\exp\left( -\tau \frac{\mathrm{d}}{\mathrm{d}t}\right) \Big)^{k}, \quad k\in \mathbb{N}_{0}.
	\end{equation}
\end{lemma}

\begin{proof}
	To prove the identity \eqref{relat}, we use the mathematical induction principle with respect to $k\in \mathbb{N}_{0} $. For $k=0$, we have $\mathcal{Q}_{1}(0)=I$. For $k=1$, using the first property of $\mathcal{Q}_{k+1}(l\tau)$ and $\mathcal{Q}_{1}(\tau)=\Theta$, we proceed as follows:
	\begin{align*}
		\sum_{l=0}^{1}\mathcal{Q}_{2}(l\tau)\exp\left( -l\tau\frac{\mathrm{d}}{\mathrm{d}t}\right) &=\sum_{l=0}^{1}A_{0}\mathcal{Q}_{1}(l\tau)\exp\left( -l\tau\frac{\mathrm{d}}{\mathrm{d}t}\right)\\ &+\sum_{l=1}^{1}A_{1}\mathcal{Q}_{1}((l-1)\tau)\exp\left( -l\tau\frac{\mathrm{d}}{\mathrm{d}t}\right) \\
		&=A_{0}\mathcal{Q}_{1}(0)+
		A_{1}\mathcal{Q}_{1}(\tau)\exp\left( -\tau\frac{\mathrm{d}}{\mathrm{d}t}\right) +A_{1}\mathcal{Q}_{1}(0)\exp\left( -\tau\frac{\mathrm{d}}{\mathrm{d}t}\right) \\
		&=A_{0}+A_{1}\exp\left( -\tau\frac{\mathrm{d}}{\mathrm{d}t}\right) .
	\end{align*}
	Then, assuming that the identity \eqref{relat} is true for the case of $k=n$:
	\begin{equation}\label{n}
		\sum_{l=0}^{n}\mathcal{Q}_{n+1}(l\tau)\exp\left(-l\tau \frac{\mathrm{d}}{\mathrm{d}t}\right) =\Big(A_{0}+A_{1}\exp\left( -\tau \frac{\mathrm{d}}{\mathrm{d}t}\right) \Big)^{n}, \quad n\in \mathbb{N}_{0}.
	\end{equation}
	and we prove it for $k=n+1$ by using \eqref{n}:
	\begin{align*}
		\Big(A_{0}+A_{1}\exp\left( -\tau \frac{\mathrm{d}}{\mathrm{d}t}\right)\Big)^{n+1}&=\Big(A_{0}+A_{1}\exp\left(-\tau \frac{\mathrm{d}}{\mathrm{d}t}\right)\Big)\Big(A_{0}+A_{1}\exp\left( -\tau \frac{\mathrm{d}}{\mathrm{d}t}\right) \Big)^{n}\\
		&=\Big(A_{0}+A_{1}\exp\left(-\tau \frac{\mathrm{d}}{\mathrm{d}t}\right) \Big)\sum_{l=0}^{n}\mathcal{Q}_{n+1}(l\tau)\exp\left( -l\tau \frac{\mathrm{d}}{\mathrm{d}t}\right) \\
		&=\sum_{l=0}^{n}A_{0}\mathcal{Q}_{n+1}(l\tau)\exp\left( -l\tau \frac{\mathrm{d}}{\mathrm{d}t}\right)\\ &+\sum_{l=0}^{n}A_{1}\mathcal{Q}_{n+1}(l\tau)\exp\left( -(l+1)\tau \frac{\mathrm{d}}{\mathrm{d}t}\right) \\
		&=\sum_{l=0}^{n}A_{0}\mathcal{Q}_{n+1}(l\tau)\exp\left( -l\tau \frac{\mathrm{d}}{\mathrm{d}t}\right)\\
		&+\sum_{l=1}^{n+1}A_{1}\mathcal{Q}_{n+1}((l-1)\tau)\exp\left( -l\tau \frac{\mathrm{d}}{\mathrm{d}t}\right) \\
		&=A_{0}^{n+1}+A_{1}^{n+1}\exp\left( -(n+1)\tau \frac{\mathrm{d}}{\mathrm{d}t}\right) \\
		&+\sum_{l=1}^{n}\Big(A_{0}\mathcal{Q}_{n+1}(l\tau)+A_{1}\mathcal{Q}_{n+1}((l-1)\tau)\Big)\exp\left( -l\tau \frac{\mathrm{d}}{\mathrm{d}t}\right) \\
		&=\sum_{l=0}^{n+1}\Big(A_{0}\mathcal{Q}_{n+1}(l\tau)+A_{1}\mathcal{Q}_{n+1}((l-1)\tau)\Big)\exp\left( -l\tau \frac{\mathrm{d}}{\mathrm{d}t}\right) \\
		&=\sum_{l=0}^{n+1}\mathcal{Q}_{n+2}(l\tau)\exp\left( -l\tau \frac{\mathrm{d}}{\mathrm{d}t}\right) .
	\end{align*}
	Therefore, the identity \eqref{relat} holds true for any $k\in \mathbb{N}_{0}$. The proof is complete.
\end{proof}
\begin{corollary}
	The delayed analogue of the \textit{ binomial theorem} for permutable linear operators $A_{0},A_{1} \in \mathcal{L}\left( X\right)$ is defined by
	\begin{equation*}
		\sum_{l=0}^{k}\binom{k}{l}A_{0}^{k-l}A_{1}^{l}\exp\left(-l\tau \frac{\mathrm{d}}{\mathrm{d}t}\right) =\Big(A_{0}+A_{1}\exp\left( -\tau \frac{\mathrm{d}}{\mathrm{d}t}\right) \Big)^{k}, \quad k\in \mathbb{N}_{0}.
	\end{equation*}
\end{corollary}
\begin{remark}
	The multivariate analogue of this identity is proved for the representation formula of a multi-delayed perturbation of Mittag-Leffler type matrix function generated by $A_{0}$, $A_{1}$, \ldots, $A_{n}\in\mathbb{R}^{n \times n}$ which is a fundamental solution of the fractional-order multi-delay system in \cite{jmaa}.
\end{remark}

Therefore, by using the formulae \eqref{shift} and \eqref{relat} together, we get a new explicit formula for the delayed perturbation of an operator-valued exponential function as below.
\begin{definition}
	The delayed perturbation of an operator-valued exponential function $\mathcal{S}(\cdot;\tau):\mathbb{R}\to  \mathcal{L}\left(X\right)$ generated by linear operators $A_{0},A_{1} \in\mathcal{L}\left(X\right) $ is defined by
	\begin{equation}\label{label}
		\mathcal{S}(t;\tau)=\sum_{k=0}^{\infty}\Big(A_{0}+A_{1}\exp\left( -\tau \frac{\mathrm{d}}{\mathrm{d}t}\right) \Big)^{k}
		\frac{(t)_{+}^{k}}{k!}, \quad t \in \mathbb{R}.
	\end{equation}
\end{definition}
\begin{remark}
	What is the advantage of this alternative definition? This definition retains a \textit{delayed analogue} of \textit{binomial formula}. Using this representation formula \eqref{label} we can easily determine the binomial expansion of commutative and non-commutative linear bounded operators corresponding to the domain of \textit{piece-wise} defined delayed operator-valued functions.
\end{remark}
\begin{remark}
	Note that, in the case of $\tau=0$, the formula \eqref{label} turns into the following \textit{perturbed} operator-valued exponential function:
	\begin{equation*}
		\mathcal{S}(t)=\sum_{k=0}^{\infty}\Big(A_{0}+A_{1}\Big)^{k}
		\frac{t^{k}}{k!}=\exp\left(A_{0}+A_{1}\right)t , \quad t \in \mathbb{R}.
	\end{equation*}
\end{remark}

The following theorem refers to a fundamental solution of the delay evolution equation \eqref{main equation} with permutable bounded linear operators. Since $A_{0}A_{1}=A_{1}A_{0}$, the fundamental solution is the product of exponential and delayed exponential operator-valued functions.

\begin{theorem}\label{permutable}
	Let $A_{0},A_{1} \in \mathcal{L}(X)$ with zero commutator $[A_{0},A_{1}]\coloneqq A_{0}A_{1}-A_{1}A_{0}=0$. Then, the fundamental solution $\mathcal{S}(\cdot;\tau): [-\tau,T] \to \mathcal{L}\left(X\right)$  of linear homogeneous evolution equation \eqref{main equation} with a discrete delay which is satisfying initial conditions \eqref{data} for $t\in [-\tau,0]$ can be expressed by
	\begin{align}\label{formula}
		&\mathcal{S}(t;\tau)=\exp\left( A_{0}t\right)\exp_{\tau}\left( A_{2}t\right)=\exp_{\tau}\left( A_{2}t\right)\exp\left( A_{0}t\right),\nonumber\\   &A_{2}\coloneqq A_{1}\exp\left( -A_{0}\tau\right) \in \mathcal{L}\left( X\right), \hspace{0.1 cm} n\tau < t\leq  (n+1)\tau, \quad n\in \mathbb{N}_{0}.
	\end{align}
\end{theorem}
\begin{proof}
	Since bounded linear operators $A_{0},A_{1}\in\mathcal{L}\left(X \right)$ are permutable, i.e.  $A_{0}A_{1}=A_{1}A_{0}$, we have $\exp\left( A_{0}(t-s)\right) A_{1}=A_{1}\exp\left( A_{0}(t-s)\right)$. By making use of this relation, we derive the following cases, recursively.
	Firstly, for $n=0$, we attain that
	\begin{equation*}
		S_{0}(t,0)=\exp\left( A_{0}t\right) , \quad t\geq 0.
	\end{equation*}
	For the case of $n=1$, we have
	\begin{align*}
		S_{1}(t,\tau)&=\int_{\tau}^{t}\exp\left( A_{0}(t-s)\right) A_{1}\exp\left( A_{0}(s-\tau)\right) \mathrm{d}s\\
		&=\exp\left( A_{0}(t-\tau)\right) A_{1}(t-\tau)\\
		&=\exp\left( A_{0}t\right) A_{1}\exp\left( -A_{0}\tau\right) (t-\tau), \quad t\geq \tau.
	\end{align*}
	Similarly, for $n=2$, it follows that
	\begin{align*}
		S_{2}(t,2\tau)&=\int_{2\tau}^{t}
		\exp\left( A_{0}(t-s)\right) A_{1}S_{1}(s-\tau,\tau)\mathrm{d}s\\
		&=\int_{2\tau}^{t}\exp\left( A_{0}(t-s)\right) A_{1}\exp\left( A_{0}(s-2\tau)\right) A_{1}(s-2\tau)\mathrm{d}s\\
		&=\exp\left( A_{0}(t-2\tau)\right) A_{1}^{2}\int_{2\tau}^{t}(s-2\tau)\mathrm{d}s\\
		&=\exp\left( A_{0}t\right) A_{1}^{2}\exp\left( -2A_{0}\tau\right) \frac{(t-2\tau)^{2}}{2!}, \quad t\geq 2\tau.
	\end{align*}
	\allowdisplaybreaks
	In a recursive way, for the general $n$-th case, we acquire
	\begin{align*}
		S_{n}(t,n\tau)
		=\exp\left( A_{0}t\right) A_{1}^{n}\exp\left( -nA_{0}\tau\right) \frac{(t-n\tau)^{n}}{n!}, \quad t\geq n\tau, \quad n\in \mathbb{N}_{0}.
	\end{align*}
	
	Therefore, the fundamental solution $\mathcal{S}(\cdot;\tau):[-\tau,T]\to \mathcal{L}\left(X \right) $ of \eqref{main equation} with initial conditions \eqref{data} is a summation of a finite number above terms and we get the following result:
	\begin{align*}
		\mathcal{S}(t;\tau)=\sum_{l=0}^{n}S_{l}(t,l\tau)&=\exp\left( A_{0}t\right) \sum_{l=0}^{n}
		A_{1}^{l}\exp\left( -lA_{0}\tau\right) \frac{(t-l\tau)^{l}}{l!}\\
		&=\sum_{l=0}^{n}
		A_{1}^{l}\exp\left( -lA_{0}\tau\right) \frac{(t-l\tau)^{l}}{l!}\exp\left( A_{0}t\right)\\
		&=\exp\left( A_{0}t\right) \exp_{\tau}\left( A_{2}t\right)=\exp_{\tau}\left( A_{2}t\right)\exp\left( A_{0}t\right),\\
		A_{2}&=A_{1}\exp\left( -A_{0}\tau\right) \in \mathcal{L}\left(X \right), \quad  n\tau < t\leq  (n+1)\tau, \quad n\in \mathbb{N}_{0}.
	\end{align*}
	The proof is complete.
\end{proof}

\begin{remark}
	Alternatively, we can prove the formula \eqref{formula} directly with the result of Corollary \ref{cor}.
	By using the formula \eqref{exp}, we derive the following result: 
	\begin{align*}
		\mathcal{S}(t;\tau)&=\sum\limits_{l=0}^{n}\Big[A_{1}\left( \exp\left(A_{0}\tau\right)\right) ^{-1} \Big]^{l}\frac{(t-l\tau)^{l}}{l!}\exp\left( A_{0}t\right)\\
		&=\sum\limits_{l=0}^{n}\Big[A_{1}\exp\left( -A_{0}\tau\right) \Big]^{l}\frac{(t-l\tau)^{l}}{l!}\exp\left( A_{0}t\right)\\
		&=\exp\left( A_{0}t\right)\sum\limits_{l=0}^{n}\Big[A_{1}\exp\left( -A_{0}\tau\right) \Big]^{l}\frac{(t-l\tau)^{l}}{l!}\\
		&=\exp_{\tau}\left(A_{2}t\right)\exp\left(A_{0}t\right)=\exp\left(A_{0}t\right) \exp_{\tau}\left(A_{2}t\right),\\ A_{2}&=A_{1}\exp\left( -A_{0}\tau\right) \in \mathcal{L}\left(X \right), \quad  n\tau < t\leq  (n+1)\tau, \quad n\in \mathbb{N}_{0}.
	\end{align*}	
\end{remark}

\begin{remark}
	In addition, we can also prove the formula \eqref{formula} by using the formulae \eqref{formul} and \eqref{two} together, as follows:
	\begin{align*}
		\mathcal{S}(t;\tau)&=
		\sum_{k=l}^{\infty}\sum_{l=0}^{n}\mathcal{Q}_{k+1}(l\tau)\frac{(t-l\tau)^{k}}{k!}\\
		&=\sum_{k=0}^{\infty}\mathcal{Q}_{k+1}(0)\frac{t^{k}}{k!}+\sum_{k=1}^{\infty}\mathcal{Q}_{k+1}(\tau)\frac{(t-\tau)^{k}}{k!}+\ldots+\sum_{k=n}^{\infty}\mathcal{Q}_{k+1}(n\tau)\frac{(t-n\tau)^{k}}{k!}\\
		&=\sum_{k=0}^{\infty}A_{0}^{k}\frac{t^{k}}{k!}+\sum_{k=1}^{\infty}\binom{k}{1}A_{0}^{k-1}A_{1}\frac{(t-\tau)^{k}}{k!}+\ldots+\sum_{k=n}^{\infty}\binom{k}{n}A_{0}^{k-n}A_{1}^{n}\frac{(t-n\tau)^{k}}{k!}\\
		&=\sum_{k=0}^{\infty}A_{0}^{k}\frac{t^{k}}{k!}+\sum_{k=0}^{\infty}\binom{k+1}{1}A_{0}^{k}A_{1}\frac{(t-\tau)^{k+1}}{(k+1)!}+\ldots+\sum_{k=0}^{\infty}\binom{k+n}{n}A_{0}^{k}A_{1}^{n}\frac{(t-n\tau)^{k+n}}{(k+n)!}\\
		&=\sum_{k=0}^{\infty}A_{0}^{k}\frac{t^{k}}{k!}+ \sum_{k=0}^{\infty}A_{0}^{k}\frac{(t-\tau)^{k}}{k!}A_{1}(t-\tau)+\ldots+\sum_{k=0}^{\infty}A_{0}^{k}\frac{(t-n\tau)^{k}}{k!}A_{1}^{n}\frac{(t-n\tau)^{n}}{n!}\\
		&=\exp\left(A_{0}t\right)+\exp\left(A_{0}(t-\tau)\right)A_{1}(t-\tau)+\ldots+\exp\left(A_{0}(t-n\tau)\right)A_{1}^{n}\frac{(t-n\tau)^{n}}{n!}\\
		&=\exp\left(A_{0}t\right)\Big[I+A_{1}\exp\left(-A_{0}\tau\right)(t-\tau)
		+\ldots+\Big(A_{1}\exp\left(-A_{0}\tau\right)\Big)^{n}\frac{(t-n\tau)^{n}}{n!}\Big]\\
		&=\exp\left( A_{0}t\right) \sum_{l=0}^{n}\Big[A_{1}\exp\left( -A_{0}\tau\right) \Big]^{l}\frac{(t-l\tau)^{l}}{l!}\\
		&=\sum\limits_{l=0}^{n}\Big[A_{1}\exp\left( -A_{0}\tau\right) \Big]^{l}\frac{(t-l\tau)^{l}}{l!}\exp\left(A_{0}t\right)\\
		&=\exp\left(A_{0}t\right) \exp_{\tau}\left(A_{2}t\right)= \exp_{\tau}\left(A_{2}t\right)\exp\left(A_{0}t\right),\\ A_{2}&=A_{1}\exp\left( -A_{0}\tau\right) \in \mathcal{L}\left(X \right), \quad  n\tau < t\leq  (n+1)\tau, \quad n\in \mathbb{N}_{0}.
	\end{align*}
\end{remark}	

\begin{remark}
	It should be noted that this particular case is a natural extension of the results are attained in \cite{Khusainov-ordinary-2} in terms of commutative matrix coefficients. Moreover, corresponding results are derived for fractional-order time-delay systems in \cite{huseynov-mahmudov,control}. 
\end{remark}

Directly by means of the formula \eqref{formula}, we can obtain the following special cases, which are the fundamental solutions of the evolution equation with a \textit{pure delay} and	\textit{delay-free} systems without using the above \textit{table}.
\begin{corollary}
	Let $\mathcal{S}(t;\tau)$, $t\in [-\tau,T]$ be the fundamental solution of delay evolution equation \eqref{main equation} with initial conditions \eqref{data}. Then the following assertions hold true:
	
	$(i)$ If $A_{0}=\Theta$, then the delayed perturbation of an operator-valued exponential function turns into the pure delayed operator-valued exponential function and it can be represented by
	\begin{equation}
		\mathcal{S}(t;\tau)=\exp_{\tau}\left(A_{1}t\right) =\sum_{l=0}^{n}A_{1}^{l}\frac{(t-l\tau)^{l}}{l!},  \quad n\tau< t\leq (n+1)\tau, \quad n\in \mathbb{N}_{0}.
	\end{equation}
	
	$(ii)$ If $A_{1}=\Theta$, then the delayed perturbation of an operator-valued exponential function becomes classical operator-valued exponential function and it can be expressed by
	\begin{equation}
		\mathcal{S}(t;\tau)=\exp\left( A_{0}t\right) =\sum_{k=0}^{\infty}A_{0}^{k}\frac{t^{k}}{k!}, \quad t\geq 0.
	\end{equation}
\end{corollary}

\begin{proof}
	The proof of this Corollary follows from directly Theorem \ref{permutable}. So, we pass over it here.
\end{proof}

\begin{remark}
	Note that the first part is a natural generalisation of pure delayed exponential matrix function which is studied by Khusainov and Shuklin in \cite{Khusainov-ordinary-1} for linear differential equation with matrix coefficient and pure delay, and  the second part is well-known \textit{uniformly} continuous semigroup in a Banach space $X$ \cite{engel-nagel}.
\end{remark}

\textbf{Uniform continuity property:}
Note that in Chapter \ref{sect:main}, we have proved that if $A_{0}$ is an infinitesimal generator of the $C_{0}$-semigroup $\left\lbrace \mathcal{T}(t)\right\rbrace_{t\geq 0}$ and $A_{1}\in \mathcal{L}\left(X\right)$, then a one-parameter family of bounded linear operators $\mathcal{S}(t;\tau)$ is \textit{continuous} with respect to \textit{strong operator topology} for each $t\in[0,\infty)$. 
Analogously, one can show that $\mathcal{S}(t;\tau)$ generated by $A_{0},A_{1}\in \mathcal{L}\left(X\right)$ is \textit{uniformly continuous} with respect to $t$ on $[0,\infty)$.
\begin{align}\label{b}
	\norm{\mathcal{S}(t;\tau)-I}&=\norm{\sum_{k=1}^{\infty}\sum_{l=0}^{k}\mathcal{Q}_{k+1}(l\tau)\frac{(t-l\tau)_{+}^{k}}{k!}}\nonumber\\
	&\leq 
	\sum_{k=1}^{\infty}\sum_{l=0}^{k}\norm{\mathcal{Q}_{k+1}(l\tau)}
	\frac{(t-l\tau)_{+}^{k}}{k!}\nonumber\\
	&\leq \sum_{k=1}^{\infty}\Big(\sum_{l=0}^{k}\binom{k}{l}\norm{A_{0}}^{k-l}\norm{A_{1}}^{l}\Big)
	\frac{t^{k}}{k!}\nonumber\\
	&=\sum_{k=1}^{\infty}\Big(\norm{A_{0}}+\norm{A_{1}}\Big)^{k}\frac{t^{k}}{k!}\nonumber\\
	&=\exp\left( \Big(\norm{A_{0}}+\norm{A_{1}}\Big)t\right)-I, \quad t\geq0,
\end{align}
where we have used that
\begin{equation*}
	\norm{\mathcal{Q}_{k+1}(l\tau)}\leq \binom{k}{l}\norm{A_{0}}^{k-l}\norm{A_{1}}^{l},	
\end{equation*}	
and hence, it tends to zero as $t\to 0_{+}$.
Therefore, $\lim\limits_{t \to 0_{+}} \norm{\mathcal{S}(t;\tau)-I}=0$ and the one-parameter operator family $\left\lbrace \mathcal{S}(t;\tau) \right\rbrace_{t\geq 0}$ is uniformly continuous with respect to operator norm $\norm{\cdot}$ associated with $X$. 

For \textit{delay-free time evolution linear autonomous systems}, \textit{uniformly continuous} one-parameter family of bounded linear operators  $\left\lbrace \mathcal{S}(t)\right\rbrace _{t\geq 0}$ equals to the \textit{perturbed} operator-valued exponential function $ \left\lbrace \exp\left((A_{0}+A_{1})t\right)\right\rbrace_{t\geq 0}$ which is satisfying the \textit{semigroup} property. To show this, we use Cauchy product formula for double series as follows:
\begin{align*}
	\mathcal{S}(t)\mathcal{S}(s)&=\sum_{k=0}^{\infty}\Big(A_{0}+A_{1}\Big)^{k}
	\frac{t^{k}}{k!}\sum_{n=0}^{\infty}\Big(A_{0}+A_{1} \Big)^{n}
	\frac{s^{n}}{n!}\\
	&=\sum_{k=0}^{\infty}\sum_{n=0}^{k}\Big(A_{0}+A_{1}\Big)^{k-n}\Big(A_{0}+A_{1} \Big)^{n}\frac{t^{k-n}}{(k-n)!}\frac{s^{n}}{n!}\\
	&=\sum_{k=0}^{\infty}\left( \sum_{n=0}^{k}\binom{k}{n}
	t^{k-n}s^{n}\right) \frac{\Big(A_{0}+A_{1}\Big)^{k}}{k!}\\
	&=\sum_{k=0}^{\infty}\Big(A_{0}+A_{1}\Big)^{k}\frac{( t+s)^{k}}{k!}=\mathcal{S}(t+s), \quad t,s\geq 0.
\end{align*}

However, in our case, the same property is not working for the operator family of bounded linear operators $ \left\lbrace \mathcal{S}(t;\tau)\right\rbrace_{t\geq 0}$ which is \textit{continuous} with respect to \textit{uniform operator topology}. To show this, we will use the following \textit{counterexample}:

\textbf{Counterexample.}
Let $X=\mathbb{R}^{2}$, $T=2$, $\tau=1$, $A_{0}=\begin{pmatrix}1 & 0 \\0 & 1\end{pmatrix}$ and  $A_{1}=\begin{pmatrix}-1 & 0 \\0 & -1\end{pmatrix}$. It is obvious that $A_{0}A_{1}=A_{1}A_{0}$. Then, by elementary calculations, we obtain  $\exp\Big(A_{0}t\Big)=\begin{pmatrix}e^{t} & 0 \\0 & e^{t}\end{pmatrix}$ for any $t\in [0,2]$ and $A_{2}=A_{1}\exp\Big(-A_{0}\Big)=\begin{pmatrix}-1/e & 0 \\0 & -1/e\end{pmatrix}$. In addition, by \eqref{pure-delayed}, the \textit{pure delayed} exponential matrix function $\exp_{1}\Big(A_{2}t\Big):[-1,2]\to \mathbb{R}^{n\times n}$  is defined by
\begin{equation}\label{pure}
	\exp_{1}\Big(A_{2}t\Big)=\begin{cases}
		\Theta, \quad -1\leq t<0,\\
		I, \quad 0\leq t\leq 1,\\
		I+A_{2}(t-1), \quad 1<t\leq 2.
	\end{cases}
\end{equation}
By making use of the formula \eqref{pure}, for a \textit{left-hand side}, we derive that
\begin{align*}
	\mathcal{S}(t;\tau)\mathcal{S}(s;\tau)&=\exp\Big(A_{0}t\Big) \exp_{1}\Big(A_{2}t\Big) \exp\Big(A_{0}s\Big) \exp_{1}\Big(A_{2}s\Big)\\
	&=\exp\Big(A_{0}(t+s)\Big) \exp_{1}\Big(A_{2}t\Big) \exp_{1}\Big(A_{2}s\Big)\\
	&=\exp\Big(A_{0}(t+s)\Big)\Big[I+A_{2}(t-1)\Big] \Big[I+A_{2}(s-1)\Big]\\
	&=\exp\Big(A_{0}(t+s)\Big)\Big[I+A_{2}(t+s-2)+A_{2}^{2}(ts-t-s-1)\Big]\\
	&=\begin{pmatrix}e^{t+s} & 0 \\0 & e^{t+s}\end{pmatrix}\begin{pmatrix}\frac{-ts+(t+s)(1-e)+2e+2}{e^{2}} & 0 \\0 & \frac{-ts+(t+s)(1-e)+2e+2}{e^{2}}\end{pmatrix}\\
	&\hspace{-3cm}=\begin{pmatrix}e^{t+s-2}\Big[-ts+(t+s)(1-e)+2e+2\Big] & 0 \\0 & e^{t+s-2}\Big[-ts+(t+s)(1-e)+2e+2\Big] \end{pmatrix}, \quad \forall t,s\in[0,2].
\end{align*}
Similarly, for a \textit{right-hand side}, we have:	
\begin{align*}
	\mathcal{S}(t+s;\tau)&=\exp\Big(A_{0}(t+s)\Big) \exp_{1}\Big(A_{2}(t+s)\Big)\\
	&=\exp\Big(A_{0}(t+s)\Big)\Big[I+A_{2}(t+s-1)\Big]\\
	&=\begin{pmatrix}e^{t+s} & 0 \\0 & e^{t+s}\end{pmatrix}\begin{pmatrix}\frac{-(t+s)+e+1}{e} & 0 \\0 & \frac{-(t+s)+e+1}{e}\end{pmatrix}\\
	&=\begin{pmatrix}e^{t+s-1}\Big[-(t+s)+e+1\Big] & 0 \\0 & e^{t+s-1}\Big[-(t+s)+e+1\Big] \end{pmatrix}, \quad \forall t,s\in[0,2].
\end{align*}
It follows that, $\mathcal{S}\left( t;\tau\right) \mathcal{S}\left( s;\tau\right)\neq \mathcal{S}\left( t+s;\tau\right)$ for any $t,s\in[0,2]$. 	

Furthermore, we can say that the similar scenario holds true for the \textit{strongly continuous} family of bounded linear operators  $\left\lbrace \mathcal{S}(t;\tau)\right\rbrace _{t\geq 0}$ (see Section \ref{sect:main}), which does not satisfy the \textit{semigroup property}, i.e., $\mathcal{S}\left( t;\tau\right) \mathcal{S}\left( s;\tau\right)\neq \mathcal{S}\left( t+s;\tau\right)$ for any $t,s \geq 0$.

\section{Application: a delayed heat equation }\label{sect:partial}
Let us take $X=\mathbb{L}^{2}\left( [0,\pi],\mathbb{R}\right) $. We consider the following initial-boundary value problem with homogeneous Dirichlet boundary conditions for a one-dimensional heat equation with a constant delay:
\begin{equation}\label{part-1}
	\begin{cases}
		\frac{\partial}{\partial t}u(x,t)=a^{2}\frac{\partial^{2}}{\partial x^{2}}u(x,t)+bu(x,t-\tau)+\psi(x,t),\quad x \in [0,\pi], \quad t\geq 0,\\
		u(x,t)=\varphi(x,t),\quad  x\in [0,\pi], \quad  t\in[-\tau,0],\\
		u(0,t)=u(\pi,t)=0, \quad t \geq -\tau,
	\end{cases}
\end{equation}
where $a,b \in \mathbb{R}$ and $\tau>0$. 

We define the following unbounded linear operator $A_{0}:\mathcal{D}\left( A_{0}\right) \subseteq X\to X$ as follows:
\begin{equation*}
	A_{0}u=a^{2}\frac{\partial^{2}}{\partial x^{2}}u, \quad a \in \mathbb{R}, \quad  u\in \mathcal{D}\left(A_{0}\right),
\end{equation*} 
with the domain is given by
\begin{equation*}
	\mathcal{D}\left(A_{0}\right) =\left\lbrace u\in X:  u,\frac{\partial}{\partial x}u \hspace{0.2 cm} \text{are absolutely continuous}, \frac{\partial^{2}}{\partial x^{2}}u \in X, \hspace{0.2 cm} u(0)=u(\pi)=0 \right\rbrace.
\end{equation*}
Moreover, a linear bounded operator $A_{1}: X\to X$ is defined by $A_{1}u=bu$ for all $b\in \mathbb{R}$ and $x\in X$.

It is known that $A_{0}$ has a discrete spectrum with eigenvalues of the form $\lambda_{n}=-a^{2}n^{2}$, $n\in \mathbb{N}$ and the corresponding normalized eigenvectors are given by $u_{n}(x)\coloneqq \sqrt{\frac{2}{\pi}}\sin(nx)$, $x\in [0,\pi]$ for any $n\in \mathbb{N}$. Moreover, $\left\lbrace u_{n}: n\in \mathbb{N} \right\rbrace$ is an \textit{orthonormal basis} for $X$ and thus, $A_{0}$ and $A_{1}$  has the following \textit{spectral representations}:
\begin{align*}
	&A_{0}u=\sum\limits_{n=1}^{\infty}-a^{2}n^{2}\langle u,u_{n}\rangle u_{n},\quad a \in \mathbb{R}, \quad  u\in \mathcal{D}\left(A_{0}\right) ,\\
	&A_{1}u=\sum\limits_{n=1}^{\infty}b\langle u,u_{n}\rangle u_{n},\quad b\in \mathbb{R}, \quad  u\in X.
\end{align*}

Furthermore, $A_{0}$ is a closed, densely-defined linear operator and the infinitesimal generator of \textit{$C_{0}$-semigroup} $\left\lbrace \mathcal{T}(t)\right\rbrace_{t \geq 0}$ on $X$ which is represented by
\begin{equation*}
	\mathcal{T}(t)u=\sum\limits_{n=1}^{\infty}\exp\left( -a^{2}n^{2} t\right) \langle u,u_{n}\rangle u_{n},\quad a \in \mathbb{R}, \quad  u\in \mathcal{D}\left(A_{0}\right), \quad t\geq 0.
\end{equation*}

Therefore, the initial-boundary value problem for a one-dimensional heat equation \eqref{part-1} with a constant delay $\tau>0$  can be formulated in abstract sense as follows:
\begin{equation}\label{short}
	\begin{cases}
		\frac{\mathrm{d}}{\mathrm{d}t}u(t)=A_{0}u(t)+A_{1}u(t-\tau)+\psi(t), \quad t\geq 0,\\
		u(t)=\varphi(t), \quad -\tau\leq t\leq 0,
	\end{cases}
\end{equation}
where $u(t)=u(\cdot,t)$, $\varphi(t)=\varphi(\cdot,t)$ and $\psi(t)=\psi(\cdot,t)$.

Since a strongly continuous semigroup $\left\lbrace \mathcal{T}(t)\right\rbrace _{t\geq 0}$ generated by $A_{0}=\frac{\partial^{2}}{\partial x^{2}}$ commutes with $A_{1}\in \mathcal{L}\left(X\right)$, by making use the formula \eqref{partially}, for any $u\in X$, a fundamental solution $\mathcal{S}(\cdot;\tau):[-\tau,\infty) \to \mathcal{L}\left( X\right)$ of \eqref{short} which is satisfying initial conditions $\mathcal{S}(t;\tau)=\Theta$, $-\tau \leq t <0$ and $\mathcal{S}(0;\tau)=I$ can be represented as follows:
\begin{align}\label{sol-3}
	\mathcal{S}(t;\tau)u&=\sum_{n=1}^{\infty}\sum_{k=0}^{\infty}b^{k}\mathcal{T}(t-k\tau)\frac{(t-k\tau)^{k}_{+}}{k!}\langle u,u_{n}\rangle u_{n}\nonumber\\
	&=\sum_{n=1}^{\infty}\sum_{k=0}^{\infty}
	b^{k}\exp\Big(-a^{2}n^{2}(t-k\tau)\Big)\frac{(t-k\tau)_{+}^{k}}{k!} \langle u,u_{n}\rangle u_{n}\nonumber\\
	&=\sum_{n=1}^{\infty}\sum_{k=0}^{\infty}
	\exp\Big(-a^{2}n^{2}t\Big)\Big[b\exp(a^{2}n^{2}\tau)\Big]^{k}\frac{(t-k\tau)_{+}^{k}}{k!} \langle u,u_{n}\rangle u_{n}\nonumber\\
	&=\sum_{n=1}^{\infty}
	\exp\Big(-a^{2}n^{2}t\Big)\exp_{\tau}\Big(b_{n}t\Big) \langle u,u_{n}\rangle u_{n}, \quad t\geq 0,
\end{align}
where $\exp_{\tau}(\cdot):[0,\infty)\to \mathbb{R}$ is a \textit{pure delayed} operator-valued exponential function defined by
\begin{equation*}
	\exp_{\tau}\Big(b_{n}t\Big) =\sum_{k=0}^{\infty}b_{n}^{k}\frac{\Big(t-k\tau\Big)_{+}^{k}}{k!}, \quad b_{n}=b\exp\left(a^{2} n^{2}\tau\right), \quad n\in \mathbb{N}, \quad t\geq 0.
\end{equation*}

Therefore, by using the formulas \eqref{sol-2} and \eqref{sol-3}, the \textit{classical} or \textit{strong} solution $u(\cdot)\in \mathbb{C}\left([-\tau,\infty),X\right)\cap\mathbb{C}^{1}\left([0,\infty),X\right) $  of abstract Cauchy problem \eqref{short} which is satisfying $u(t)\in\mathcal{D} \left(A_{0}\right)$ for all $t\geq 0$ with $\varphi(0)\in \mathcal{D}\left( A_{0}\right)$ can be expressed with the help of a special case of \textit{delayed Dyson-Phillips series} as below:
\begin{align}\label{son}
	u(t)&=\sum_{n=1}^{\infty}\Biggl\{\exp\left( -a^{2}n^{2}t\right) \exp_{\tau}\left( b_{n}t\right) <\varphi(0),u_{n}>\nonumber\\
	&+b\int_{-\tau}^{0}\exp\left( -a^{2}n^{2}(t-\tau-s)\right) \exp_{\tau}\left( b_{n}(t-\tau-s)\right) <\varphi(s),u_{n}>\mathrm{d}s\nonumber\\
	&+\int_{0}^{t}
	\exp\left( -a^{2}n^{2}(t-s)\right) \exp_{\tau}\left( b_{n}(t-s)\right) <\psi(s),u_{n}>\mathrm{d}s\Biggr\}u_{n}, \quad  t\geq  0.
\end{align}

Similarly, with the help of formulae \eqref{sol-1} and \eqref{sol-3}, the classical (strong)  solution $u(\cdot)\in \mathbb{C}^{1}\left([-\tau,\infty),X\right)$  of abstract Cauchy problem \eqref{short} which is satisfying $u(t)\in\mathcal{D} \left(A_{0}\right)$ for all $t\geq 0$ with $\varphi(t)\in \mathcal{D}\left( A_{0}\right)$ for any $t\in[-\tau,0]$ can also be expressed by
\begin{align}\label{son-1}
	u(t)&=\sum_{n=1}^{\infty}\Biggl\{\exp\left( -a^{2}n^{2}(t+\tau)\right) \exp_{\tau}\left( b_{n}(t+\tau)\right) <\varphi(-\tau),u_{n}>\nonumber\\
	&+\int_{-\tau}^{0}\exp\left( -a^{2}n^{2}(t-s)\right) \exp_{\tau}\left( b_{n}(t-s)\right) <\varphi^{\prime}(s)+a^{2}n^{2}\varphi(s),u_{n}>\mathrm{d}s\nonumber\\
	&+\int_{0}^{t}\exp\left( -a^{2}n^{2}(t-s)\right) \exp_{\tau}\left( b_{n}(t-s)\right) <\psi(s),u_{n}>\mathrm{d}s\Biggr\}u_{n},\quad t\geq -\tau.
\end{align}

\begin{remark}
	Currently, in \cite{jde}, Pinto et al. have studied an approximation of a mild solution $u(\cdot) \in X=\mathbb{L}^{2}\left( [0,\pi],\mathbb{R}\right)$ of first-order linear homogeneous abstract differential problem \eqref{short} with a constant delay (where $\psi(t)=0$, $a=1$, $\tau=1$, $A_{1}u=Lu$ with $L \geq 0$), which depends on an initial history condition $\varphi(t)\in X$ for $-1\leq t\leq 0$ and unbounded closed linear operator $A_{0}=\frac{\partial^{2}}{\partial x^{2}}$ generating a $C_{0}$-semigroup on a Banach space $X$. The authors have introduced in \cite{jde}, the mild solution $u(t)$ of the abstract Cauchy problem \eqref{short} associated with $\varphi(t)$, $-1 \leq t \leq 0$, for any $t\in[0,1]$ is given by
	\allowdisplaybreaks
	\begin{align*}
		u(t)&=\sum_{n=1}^{\infty}\exp\left( -n^{2}t\right) <\varphi(0),u_{n}>u_{n}\\
		&+L\sum_{n=1}^{\infty}\frac{\exp\left(-n^{2}t\right) +n^{2}t-1}{n^{4}}<\varphi(0),u_{n}>u_{n}, \quad t\geq 0.
	\end{align*}
	In contrast to \cite{jde}, we can easily deduce a \textit{mild} solution $u(\cdot)\in \mathbb{C}\left([-\tau,\infty),X\right)\cap\mathbb{C}^{1}\left([0,\infty),X\right) $  of \eqref{short} by taking $a=\tau=1$ in \eqref{son} as follows:
	\begin{align*}
		u(t)&=\sum_{n=1}^{\infty}\Biggl\{\exp\left(-n^{2}t\right) \exp_{1}\left( b_{n}t\right) <\varphi(0),u_{n}>\\
		&+L\int_{-1}^{0}\exp\left( -n^{2}(t-1-s)\right) \exp_{1}\left( b_{n}(t-1-s)\right) <\varphi(s),u_{n}>\mathrm{d}s\nonumber\\
		&+\int_{0}^{t}
		\exp\left( -n^{2}(t-s)\right) \exp_{1}\left( b_{n}(t-s)\right) <\psi(s),u_{n}>\mathrm{d}s\Biggr\}u_{n}, \quad  t\geq  0.
	\end{align*}
	Furthermore, an exact explicit formula of a mild solution $u(\cdot)\in \mathbb{C}^{1}\left([-\tau,\infty),X\right) $ of \eqref{short} can be represented by
	\begin{align*}
		u(t)&=\sum_{n=1}^{\infty}\Biggl\{\exp\left( -n^{2}(t+1)\right) \exp_{1}\left( b_{n}(t+1)\right) <\varphi(-1),u_{n}>\\
		&+\int_{-1}^{0}\exp\left( -n^{2}(t-s)\right) \exp_{1}\left( b_{n}(t-s)\right) <\varphi^{\prime}(s)+n^{2}\varphi(s),u_{n}>\mathrm{d}s\\
		&+\int_{0}^{t}\exp\left( -n^{2}(t-s)\right) \exp_{\tau}\left( b_{n}(t-s)\right) <\psi(s),u_{n}>\mathrm{d}s\Biggr\}u_{n},\quad t\geq -1.
	\end{align*}
	where
	\begin{equation*}
		\exp_{1}\left(b_{n}t\right) =\sum_{k=0}^{\infty}b_{n}^{k}\frac{(t-k)_{+}^{k}}{k!}, \quad b_{n}=L\exp\left( n^{2}\right) , \quad n\in \mathbb{N}, \quad t\geq 0.
	\end{equation*}
\end{remark}
\allowdisplaybreaks
\begin{remark}
	Furthermore, one can use Fourier method for the construction of the first boundary value problem \eqref{part-1}.
	Note that the solution of one dimensional heat equation \eqref{part-1} with a discrete delay $\tau>0$ can be expressed with Fourier coefficients which is studied in \cite{Khusainov-partial-1}-\cite{Khusainov-Pokojovy} for $x\in [0,\pi]$: 
	\begin{align*}
		u(x,t)&=\sum_{n=1}^{\infty}\Biggl\{
		\exp\left( -a^{2}n^{2}(t+\tau)\right) \exp_{\tau}\left( b_{n}(t+\tau)\right) \Phi_{n}(-\tau)\\
		&+\int_{-\tau}^{0}\exp\left( -a^{2}n^{2}(t-s)\right) \exp_{\tau}\left( b_{n}(t-s)\right) \Big[\Phi_{n}^{\prime}(s)+a^{2}n^{2}\Phi_{n}(s)\Big]\mathrm{d}s\\
		&+\int_{0}^{t}\exp\left( -a^{2}n^{2}(t-s)\right) \exp_{\tau}\left( b_{n}(t-s)\right) \Psi_{n}(s)\mathrm{d}s\Biggr\}\sin(nx), \quad  t\geq -\tau,\\
	\end{align*}
	or
	\begin{align*}
		u(x,t)&=\sum_{n=1}^{\infty}\Biggl\{\exp\left( -a^{2}n^{2}t\right) \exp_{\tau}\left(b_{n}t\right) \Phi_{n}(0)\\
		&+b\int_{-\tau}^{0}\exp\left( -a^{2}n^{2}(t-\tau-s)\right) \exp_{\tau}\left( b_{n}(t-\tau-s)\right) \Phi_{n}(s)\mathrm{d}s\\
		&+\int_{0}^{t}\exp\left( -a^{2}n^{2}(t-s)\right) \exp_{\tau}\left( b_{n}(t-s)\right) \Psi_{n}(s)\mathrm{d}s\Biggr\}\sin(nx),  \quad  t\geq 0,
	\end{align*}
	where $b_{n}=b\exp\left(a^{2} n^{2}\tau\right)$, $n\in \mathbb{N}$; $\Phi_{n}(\cdot):[-\tau,0]\to\mathbb{R}$ and $\Psi_{n}(\cdot):[0,\infty)\to\mathbb{R}$ for $n\in \mathbb{N}$ are Fourier coefficients of $\varphi(x,t)$  and $\psi(x,t)$, respectively, that is 
	\begin{align*}
		\Phi_{n}(t)&=\frac{2}{\pi}\int_{0}^{\pi}\varphi(\xi,t)\sin(n\xi)\mathrm{d}\xi,\quad t\in [-\tau,0],\\
		\Psi_{n}(t)&=\frac{2}{\pi}\int_{0}^{\pi}\psi(\xi,t)\sin(n\xi)\mathrm{d}\xi, \quad t\in[0,\infty).
	\end{align*}
\end{remark}


\begin{thebibliography}{99}
		
	\bibitem{Kolmanovskii-Myshkis}
	V. Kolmanovskii, A. Myshkis, Applied theory of functional differential
	equations, Dordrecht, The Netherlands: Kluwer, 1992.
	
	\bibitem{Hale}	
	J. K. Hale, Theory of functional differential equations, Applied Mathematical Sciences Series, Vol. 3, 1977.	
	
	\bibitem{Bellman-Cooke}
	R. Bellman, K. L. Cooke, Differential-Difference Equations, Academic, New York, 1963.
	
	\bibitem{Khusainov-ordinary-1}	
	D. Ya. Khusainov, G. V. Shuklin, On relative controllability in systems with pure delay, Prikladnaya Mekhanika, Int. Appl. Mech., 41(2) (2005) 210–221.	
	
	
	
	\bibitem{Khusainov-ordinary-2}	
	D. Ya. Khusainov, G. V. Shuklin, Linear autonomous time-delay system with permutation matrices solving, Stud. Univ. Zilina, 17(1) (2003) 101–108.	
	
	\bibitem{huseynov-mahmudov}
	I. T. Huseynov, N. I. Mahmudov, Delayed analogue of three-parameter Mittag–Leffler functions and their applications to Caputo type fractional time delay differential equations, Math. Methods Appl. Sci., (2021), https://doi .org/10.1002/mma.6761.
	
	\bibitem{mahmudov}
	N. I. Mahmudov, Delayed perturbation of Mittag-Leffler functions their applications to fractional linear delay differential equations, Math. Methods Appl. Sci., (2018) 1–9, https://doi.org/10.1002/mma.5446.
	
	\bibitem{Khusainov-partial-1}	
	D. Ya. Khusainov, A. F. Ivanov, I. V. Kovarzh,  Solution of one heat equation with delay, Nonlinear Oscil., 12(2) (2009) 1–20. 
	
	\bibitem{Khusainov-Pokojovy-Racke}
	D. Ya. Khusainov, M. Pokojovy, R. Racke, Strong and mild extrapolated $L^{2}$
	-solutions to the heat equation with constant delay, SIAM J. Math. Anal.,
	47(1) (2015) 427–454.
	
	\bibitem{Khusainov-Pokojovy}
	D. Ya. Khusainov, M. Pokojovy,
	Solving the Linear 1D Thermoelasticity Equations with Pure Delay, Int. J. Math. Math. Sci., 2015, (2015) 1–11. 
	
	\bibitem{Samoilenko}	
	A. M. Samoilenko, L. M. Serheeva, Construction of global solutions of partial differential equations with deviating arguments in the time variable, J. Math. Sci., 212(4) (2016) 426-441.
	
	\bibitem{ecology}	
	A. Okubo, S. A. Levin, Diffusion and ecological problems: Modern Perspectives, Springer Verlag, New York, Berlin, Heidelberg, 2001.
	
	
	\bibitem{philips}
	R. S. Phillips, Perturbation theory for semigroups of linear operators, Trans. Am. Math. Soc., 74 (1954) 199-221.
	
	\bibitem{perturbation}
	A. Ahmadova, N. I. Mahmudov, J. J. Nieto, Exponential stability and stabilization of fractional stochastic degenerate evolution equations in a Hilbert space: Subordination principle, Evol. Equ. Control Theory, https://doi.org/10.3934/eect.2022008. 
	
	
	\bibitem{per-frac-2}
	I. T. Huseynov, A. Ahmadova, N. I. Mahmudov, Perturbation properties of fractional strongly continuous cosine and sine family operators, ERA, 30(8) (2022) 2911-2940.
	
	
	
	
	\bibitem{travis-webb}
	C. C. Travis, G. F. Webb, Perturbation of strongly continuous cosine family generators, Colloquium Mathematicae, 45(2) (1981) 277-285.	
	
	\bibitem{lutz}
	D. Lutz, On bounded time-dependent perturbations of operator cosine functions, Aequationes Mathematicae, 23 (1981) 197-203.
	
	\bibitem{bazhlekova}
	E. Bazhlekova, Perturbation properties for abstract evolution equations of fractional order, Fract. Cal. Appl. Anal., 2(4) (1999) 359-366.
	
	
	
	
	\bibitem{engel-nagel}
	K.-J. Engel, R. Nagel, One-Parameter Semigroups for Linear Evolution Equations, Vol. 194, Graduate Texts in Mathematics, Springer-Verlag, New York, 2000.
	
	\bibitem{batkai-piazzera-partial}	
	A. Bátkai, S. Piazzera, Semigroups and linear partial differential equations with delay, J. Math. Anal. Appl., 264 (2001) 1–20.
	
	
	\bibitem{batkai-book}	
	A. Bátkai, S. Piazzera, Semigroups for delay equations, Resarch Notes in Mathematics, 10 A.K. Peters: Wellesley MA, 2005.	
	
	\bibitem{jde}
	M. Pinto, F. Poblete, D. Sepúlveda, Approximation of mild solutions of delay differential equations on Banach spaces, J. Differ. Equ., 303(5) (2021) 156-182.
	
	
	
	\bibitem{arendt-batty}
	W. Arendt, C. J. K. Batty, M. Hieber, F. Neubrander, Vector-Valued Laplace Transforms and Cauchy Problems, Second Edition, Vol. 96, Monographs in Mathematics, Birkhäuser, Basel, 2010.
	
	\bibitem{pazy}
	A. Pazy, Semigroups of Linear Operators and Applications to Partial Differential Equations, Springer, Berlin, 1983.
	
	\bibitem{curtain-zwart}
	R. F. Curtain, H. J. Zwart, An Introduction to Infinite Dimensional Linear Systems Theory, Berlin: Springer-Verlag, 1995.
	
	
	\bibitem{mahmudov-almatarneh}
	N. I. Mahmudov, A. M. Almatarneh,
	Stability of Ulam–Hyers and existence of solutions for impulsive time-delay semi-linear systems with non-permutable matrices, Mathematics, 8 (2020) 1-17.
	
	
	
	\bibitem{functional-1}
	N. I. Mahmudov, A. Ahmadova, I. T. Huseynov, A new technique for solving Sobolev type fractional multi-order evolution equations, Comp. Appl. Math., 41(71) (2022) 1-35.
	
	
	
	\bibitem{functional-2} 
	I. T. Huseynov, A. Ahmadova, N. I. Mahmudov, On a study of Sobolev type fractional functional evolution equations, Math. Methods Appl. Sci., 45(9) (2022) 5002-5042.
	
	\bibitem{jmaa} 
	N. I. Mahmudov, Multi-delayed perturbation of Mittag-Leffler type matrix functions, J. Math. Appl. Anal., 505(1) (2021) 125589.
	
	
	
	\bibitem{control}
	A. Ahmadova, I. T. Huseynov, N. I. Mahmudov, Controllability of fractional stochastic delay dynamical systems, Proceed. Inst. Math. Mech. ANAS, 46(2) (2020) 294–320.
\end{thebibliography}
\end{document}